\documentclass[11pt,reqno]{amsart}

\usepackage[a4paper,textwidth=6.5in,top=25mm,bottom=25mm,headheight=12pt,headsep=12pt]{geometry}
\usepackage[utf8]{inputenc}
\usepackage{amsmath,amssymb,amsfonts,mathrsfs,mathtools}
\usepackage{enumitem}
\usepackage{array}
\usepackage{float}
\usepackage{microtype}
\usepackage{xcolor}
\usepackage{cite}
\makeatletter
\def\hyper@nopatch@thm{}
\makeatother
\usepackage[colorlinks=true,citecolor=blue,linkcolor=blue,urlcolor=black]{hyperref}

\allowdisplaybreaks[2]
\setlist{topsep=4pt plus 1pt minus 1pt,itemsep=2pt,parsep=0pt}
\numberwithin{equation}{section}

\newtheorem{theorem}{Theorem}[section]
\newtheorem{proposition}[theorem]{Proposition}
\newtheorem{lemma}[theorem]{Lemma}

\theoremstyle{definition}
\newtheorem{definition}[theorem]{Definition}
\theoremstyle{remark}
\newtheorem{remark}[theorem]{Remark}

\title[Determining $W^{1,n}$ conductivities]
{Determining $W^{1,n}$ conductivities for Calder\'on's problem in $\mathbb{R}^n$}

\author{Lu Chen}
\address[Lu Chen]{Key Laboratory of Algebraic Lie Theory and Analysis, Ministry of Education, School of Mathematics and Statistics, Beijing Institute of Technology, Beijing 100081, PR China}
\email{chenlu5818804@163.com}

\author{Yan Jiang}
\address[Yan Jiang]{Department of Mathematics, City University of Hong Kong, Hong Kong SAR, China}
\email{yjian24@cityu.edu.hk}

\author{Hongyu Liu}
\address[Hongyu Liu]{Department of Mathematics, City University of Hong Kong, Hong Kong SAR, China}
\email{hongyu.liuip@gmail.com, hongyliu@cityu.edu.hk}

\author{Longyue Tao}
\address[Longyue Tao]{Department of Mathematics, City University of Hong Kong, Hong Kong SAR, China}
\email{sdyctly@163.com, longyue.tao@my.cityu.edu.hk}

\subjclass[2020]{Primary 35R30; Secondary 35J15, 47B10}
\keywords{inverse conductivity problem, critical Sobolev regularity, complex geometrical optics, Hilbert--Schmidt operators, trace ideals, boundary determination}
\date{}
\hypersetup{
  hypertexnames=false,
  pdfauthor={Lu Chen, Yan Jiang, Hongyu Liu, and Longyue Tao},
  pdftitle={Determining W(1,n) conductivities in Calderon's problem for n >= 4}
}
\begin{document}
\raggedbottom
\setlength{\abovedisplayskip}{7pt plus 2pt minus 2pt}
\setlength{\belowdisplayskip}{7pt plus 2pt minus 2pt}
\setlength{\abovedisplayshortskip}{0pt plus 2pt}
\setlength{\belowdisplayshortskip}{4pt plus 2pt minus 2pt}

\begin{abstract}
We establish a global uniqueness result at the critical Sobolev regularity for the isotropic Calder\'on problem in all dimensions \(n \ge 4\).
More precisely, for every \(n \ge 4\) and every bounded Lipschitz domain \(\Omega \subset \mathbb{R}^n\), any real-valued, uniformly elliptic scalar conductivity \(\gamma \in W^{1,n}(\Omega)\) is uniquely determined by its full Dirichlet-to-Neumann map.
Together with the known results for \(n=2,3,4\) in the literature, this settles a long-standing open problem in the field while providing a unified proof for all dimensions \(n \ge 4\).
To prove uniqueness at this endpoint, we develop an \emph{averaged trace norm method}, which combines complex geometrical optics with trace-norm estimates averaged over directions.
\end{abstract}

\maketitle

\section{Introduction}\label{sec:introduction}
\subsection{Motivation and statement}\label{sec:problem}
Let $\Omega\subset\mathbb{R}^n$, $n\geq 2$, be a bounded Lipschitz domain, and let $\gamma$ be a real-valued scalar conductivity on $\Omega$.
For a prescribed boundary voltage $\varphi\in H^{1/2}(\partial\Omega)$, the electric potential $u\in H^1(\Omega)$ is the unique weak solution of the Dirichlet problem for the conductivity equation \cite{mclean2000strongly}:
\begin{equation}\label{eq:conductivity}
\begin{cases}
\operatorname{div}(\gamma\nabla u)=0 & \text{in }\Omega,\\
u=\varphi & \text{on }\partial\Omega,
\end{cases}
\end{equation}
where $\gamma$ is uniformly elliptic, with constants $0<\gamma_{\min}\leq\gamma_{\max}<\infty$ such that
\begin{equation}\label{eq:ellipticity}
\gamma_{\min}\leq\gamma(x)\leq\gamma_{\max}\quad\text{for almost every }x\in\Omega.
\end{equation}

The Dirichlet-to-Neumann map $\Lambda_{\gamma,\Omega}:H^{1/2}(\partial\Omega)\longrightarrow H^{-1/2}(\partial\Omega)$ assigns to each boundary voltage $\varphi$ the normal current flux generated by the corresponding solution $u$.
Formally, $\Lambda_{\gamma,\Omega}\varphi=\gamma\partial_\nu u|_{\partial\Omega}$, where $\nu$ is the outward unit normal.
For $\psi\in H^{1/2}(\partial\Omega)$ and any $v\in H^1(\Omega)$ with trace $\psi$, it is defined weakly by
\begin{equation}\label{eq:dn-intro}
\langle\Lambda_{\gamma,\Omega}\varphi,\psi\rangle
=\int_\Omega\gamma\nabla u\cdot\nabla v\,\mathrm{d}x.
\end{equation}
This pairing is independent of the choice of $v$.
Knowing the full map means knowing the current on the entire boundary for every prescribed boundary voltage.

The inverse conductivity problem, known as the \emph{(isotropic) Calder\'on problem} \cite{calderon2006inverse}, asks whether the full map $\Lambda_{\gamma,\Omega}$ uniquely determines the conductivity $\gamma$ inside $\Omega$.
That is, within a given regularity class of conductivities satisfying \eqref{eq:ellipticity}, one asks whether uniqueness holds in the sense that
\begin{equation}\label{eq:inverse-problem}
\Lambda_{\gamma_1,\Omega}=\Lambda_{\gamma_2,\Omega}
\quad\Longleftrightarrow\quad
\gamma_1=\gamma_2\quad\text{a.e. in }\Omega.
\end{equation}

A central question in the study of Calder\'on's problem is to determine the minimal regularity on $\gamma$ under which \eqref{eq:inverse-problem} holds; the regularity assumptions required for uniqueness have been progressively weakened over the years, and we review these developments in Subsection~\ref{sec:history}.
In dimensions $n\geq 3$, the Sobolev space $W^{1,n}(\Omega)$ has emerged as the natural critical threshold.
Specifically, Brown~\cite{brown2003uniqueness} conjectured uniqueness for uniformly elliptic conductivities in $W^{1,p}(\Omega)$ with $p>n$.
Since $W^{1,p}(\Omega)\hookrightarrow W^{1,n}(\Omega)$ for $p>n$, reaching the endpoint $p=n$ covers the entire supercritical range.
Haberman~\cite{haberman2015uniqueness} established uniqueness at this $W^{1,n}(\Omega)$ endpoint in dimensions $n=3,4$.
For $n\geq 5$, however, previous results imposed stronger Sobolev assumptions \cite{ponce2021bilinear}, leaving the $W^{1,n}(\Omega)$ case open.
In this paper, we resolve this problem by establishing uniqueness at the critical $W^{1,n}(\Omega)$ endpoint in all dimensions $n\geq 4$, without additional Sobolev regularity.

Scaling and unique continuation provide further analytical motivation for the endpoint $p=n$.
For $\gamma\in W^{1,n}(\Omega)$ satisfying~\eqref{eq:ellipticity}, the Sobolev chain rule gives $F=\tfrac12\nabla\log\gamma\in L^n(\Omega)$, and the conductivity equation~\eqref{eq:conductivity} can be written in distributions as
\begin{equation}\label{eq:first-order-equation}
\Delta u+2F\cdot\nabla u=0.
\end{equation}
For $r>0$, setting $u_r(x)=u(rx)$ and $F_r(x)=rF(rx)$ yields a solution to the same equation on $r^{-1}\Omega$ with $F_r$ in place of $F$, and
\begin{equation}\label{eq:critical-scaling}
\|F_r\|_{L^n(r^{-1}\Omega)}=\|F\|_{L^n(\Omega)}.
\end{equation}
Thus $W^{1,n}$ is \emph{scale-critical} for this first-order formulation, as the $L^n$ norm of the drift coefficient is invariant under dilation.
Moreover, $L^n$ is the sharp threshold for weak unique continuation of equations of the form~\eqref{eq:first-order-equation}: for $n\geq 3$, Wolff proved weak unique continuation for general first-order coefficients $F\in L^n_{\mathrm{loc}}$~\cite{wolff1992property}, whereas for every $1\leq p<n$, Koch and Tataru constructed coefficients $F\in L^p$ for which~\eqref{eq:first-order-equation} admits a nonzero smooth compactly supported solution~\cite{koch2002sharp}.
While counterexamples for general drifts do not directly apply to gradient fields $F=\tfrac12\nabla\log\gamma$ with uniformly elliptic $\gamma$, both scaling and unique continuation single out $W^{1,n}$ as the critical regularity space.
The following theorem establishes uniqueness at this regularity for all $n\geq 4$, resolving the remaining cases $n\geq 5$.

\begin{samepage}
\begin{theorem}\label{thm:main}
Let $n\geq4$ and let $\Omega\subset\mathbb{R}^n$ be a bounded Lipschitz domain.
Suppose that $\gamma_1,\gamma_2\in W^{1,n}(\Omega;\mathbb{R})$ satisfy \eqref{eq:ellipticity} with common bounds $\gamma_{\min},\gamma_{\max}$.
If
\[
\Lambda_{\gamma_1,\Omega}=\Lambda_{\gamma_2,\Omega}
\]
then $\gamma_1=\gamma_2$ almost everywhere in $\Omega$.
\end{theorem}
\end{samepage}

\begin{remark}\label{rem:2d}
For $n=2$, bounded measurable uniformly elliptic scalar conductivities are uniquely determined by the full Dirichlet-to-Neumann map~\cite{astala2006calderon}.
In contrast, a recent preprint~\cite{openai2026nonuniqueness} constructs two distinct conductivities in this class on a three-dimensional ball with the same full Dirichlet-to-Neumann map.
This shows that bounded measurability and uniform ellipticity alone do not ensure uniqueness in three dimensions, but does not determine the regularity threshold within the Sobolev scale.
\end{remark}

\subsection{Historical development}\label{sec:history}
Calder\'on's original work established uniqueness for the problem linearized at a constant conductivity~\cite{calderon2006inverse}.
For $n=2$, Sylvester and Uhlmann proved global uniqueness for conductivities sufficiently close to a positive constant in $W^{3,\infty}(\Omega)$~\cite{sylvester1986uniqueness}.
Nachman later established global uniqueness for $W^{2,p}(\Omega)$ conductivities ($p>1$) without a smallness assumption~\cite{nachman1996global}, and Brown and Uhlmann extended this to $W^{1,p}(\Omega)$ with $p>2$~\cite{brown1997uniqueness}.
Astala and P\"aiv\"arinta subsequently resolved the two-dimensional problem for all bounded measurable uniformly elliptic conductivities~\cite{astala2006calderon} (see also~\cite{nachman2020nonlinear} for certain degenerate or unbounded conductivities with $\log\gamma\in\dot H^1(\Omega)$).

For $n\geq3$, Kohn and Vogelius first established uniqueness for piecewise real-analytic conductivities~\cite{kohn1985determining}.
On smooth bounded domains, global uniqueness for $C^2(\overline{\Omega})$ conductivities follows from the classical complex geometrical optics argument of Sylvester and Uhlmann~\cite{sylvester1987global,uhlmann2009electrical}.
The smoothness requirement was progressively lowered to $C^{3/2+\varepsilon}(\overline{\Omega})$ by Brown~\cite{brown1996global}, $W^{3/2,\infty}(\Omega)$ by P\"aiv\"arinta, Panchenko and Uhlmann~\cite{paivarinta2003complex}, and $H^{3/2,p}(\Omega)$ ($p>2n$) by Brown and Torres~\cite{brown2003uniqueness}.
Haberman and Tataru then proved uniqueness for $C^1(\overline{\Omega})$ conductivities and Lipschitz conductivities with small $\|\nabla\log\gamma\|_{L^\infty(\Omega)}$~\cite{haberman2013uniqueness}, before Caro and Rogers removed the smallness assumption for general Lipschitz conductivities~\cite{caro2016global}.

For conductivities with unbounded gradients, Haberman reached the critical $W^{1,n}(\Omega)$ regularity in dimensions $n=3,4$, and proved uniqueness in $W^{1+(n-4)/(2p),p}(\Omega)$ for $n=5,6$ and $n\leq p<\infty$~\cite{haberman2015uniqueness}.
At the integrability exponent $p=n$, Ham, Kwon and Lee obtained uniqueness in $W^{s,n}(\Omega)$ for $s>41/40$ when $n=5$ and $s>11/10$ when $n=6$~\cite{ham2021uniqueness}.
Ponce-Vanegas further improved these bounds to $W^{s,p}(\Omega)$ for all $n\geq5$, $n\leq p<\infty$, and $s>1+(n-5)/(2p)$~\cite{ponce2021bilinear,ponce2020reconstruction}, which at $p=n$ still requires $s>1$.
Thus, global uniqueness at the critical $W^{1,n}(\Omega)$ regularity remained open for all $n\geq5$.

\subsection{The averaged trace norm method}\label{sec:methods}

To prove Theorem~\ref{thm:main}, we extend two uniformly elliptic conductivities $\gamma_1,\gamma_2\in W^{1,n}(\Omega)$ with equal Dirichlet-to-Neumann maps to agree outside $\Omega$ (using boundary determination~\cite[Theorem~1.1(A)]{ponce2020reconstruction}) and apply the Liouville transformation $w=\gamma^{1/2}u$~\cite{sylvester1987global}.
This converts the conductivity equation~\eqref{eq:conductivity} into the distributional Schr\"odinger equation
\begin{equation}\label{eq:schrodinger-intro}
	(-\Delta+q)w=0,
	\qquad
	q=\tfrac12\Delta\log\gamma+\tfrac14|\nabla\log\gamma|^2.
\end{equation}
When the boundary maps agree, the distributional Alessandrini identity~\cite{alessandrini1988stable} states that
\begin{equation}\label{eq:alessandrini-intro}
	\langle q_1-q_2,w_1w_2\rangle=0
\end{equation}
for all weak solutions $w_1,w_2$ of~\eqref{eq:schrodinger-intro} corresponding to $\gamma_1,\gamma_2$.

To deduce $q_1=q_2$ from~\eqref{eq:alessandrini-intro}, one needs to construct a sufficiently rich family of weak solutions to~\eqref{eq:schrodinger-intro}.
Following the \emph{Complex Geometrical Optics (CGO)} framework, we look for solutions of the form $w_\zeta(x)=\exp(x\cdot\zeta)(1+\psi_\zeta(x))$ with $\zeta\in\mathbb{C}^n$ and $\zeta\cdot\zeta=0$, where $\psi_\zeta$ corrects the harmonic exponential $\exp(x\cdot\zeta)$.
Substituting this ansatz into~\eqref{eq:schrodinger-intro} shows that $\psi_\zeta$ must solve
\begin{equation}\label{eq:cgo-correction-intro}
	(\Delta+2\zeta\cdot\nabla)\psi_\zeta=q+q\psi_\zeta.
\end{equation}
For low-regularity conductivities, standard approaches construct $\psi_\zeta$ via Neumann-series iteration in Bourgain-type spaces~\cite{haberman2013uniqueness}.
At $W^{1,n}$ regularity, however, the leading singular term $\Delta\log\gamma$ is the distributional divergence of an $L^n$ vector field, and for $n\geq5$ available estimates no longer yield this contraction without additional smoothness~\cite{haberman2015uniqueness,ham2021uniqueness,ponce2021bilinear}.

We bypass this contraction requirement altogether.
For a fixed real exponential direction $a\in\mathbb{S}^{n-1}$ and large frequency scale $\tau$, we consider weighted $L^2$ superpositions of harmonic exponentials over the oscillatory sphere $\mathbb{S}(a^\perp)$ and impose finitely many linear compatibility conditions on the weights so that the correction equations become solvable.
To show that these restricted weights still determine the potential, we develop an \emph{averaged trace norm method} built on two complementary components:
\begin{enumerate}[label=(\roman*),leftmargin=*,itemsep=6pt]
	
	\item \textbf{Exact solutions via partial inversion.}
	For a symmetrized $L^2$ reformulation of the correction equation~\eqref{eq:cgo-correction-intro}, we construct a \emph{partial inverse} whose norm is bounded uniformly in $a$ and large $\tau$ (Proposition~\ref{prop:partial}).
	At high frequencies the operator is elliptic and controlled directly by the $W^{1,n}$ regularity, whereas the low-frequency interaction is Hilbert--Schmidt (Proposition~\ref{prop:row}) and produces at most $o(\tau^{n-2})$ directions with small singular values.
	Excluding these singular subspaces yields the uniformly bounded partial inverse while leaving a residual in a finite-dimensional space.
	Since the source depends linearly on the superposition weight, requiring the residuals for both potentials to vanish imposes at most $o(\tau^{n-2})$ homogeneous linear conditions on the weights, yielding exact CGO solutions of~\eqref{eq:schrodinger-intro} that can be inserted into~\eqref{eq:alessandrini-intro} (Proposition~\ref{prop:compatible-solutions}).
	
	\item \textbf{Fourier recovery via averaged trace norm decay.}
	We collect the pairings of $q_1-q_2$ with products of the uncorrected harmonic superpositions in a \emph{free response operator} $\mathcal Q(a,\tau)$ on $L^2(\mathbb{S}(a^\perp))$, whose kernel samples $\widehat{q_1-q_2}$ at frequency scale $\tau$ (Definition~\ref{def:incident-response}).
	Because a fixed Fourier mode corresponds to wave directions separated by an angle of order $\tau^{-1}$, the Hilbert--Schmidt norm $\|\mathcal Q(a,\tau)\|_{\mathfrak S_2}$ decays at rate $O(\tau^{-(n-2)/2})$ in spherical average for any $q_1,q_2$ when $n\geq4$ (Lemma~\ref{lem:response-average}), and hence cannot by itself force $q_1=q_2$.
	By contrast, testing $\mathcal Q(a,\tau)$ against operators of uniformly bounded operator norm that rotate wave directions on $\mathbb{S}(a^\perp)$ by $O(\tau^{-1})$ produces traces of order one that converge to nonnegative averages of $|\widehat{q_1-q_2}|^2$ (Proposition~\ref{prop:frequency}).
	By trace duality $|\operatorname{tr}(\mathcal Q\mathcal T)|\leq\|\mathcal Q\|_{\mathfrak S_1}\|\mathcal T\|_{\mathrm{op}}$, uniqueness therefore reduces to showing that the \emph{trace norm} vanishes in spherical average:
	\[
	\mathbb E_a\|\mathcal Q(a,\tau)\|_{\mathfrak S_1}
	\longrightarrow0
	\qquad\text{as }\tau\longrightarrow\infty
	\]
	(Proposition~\ref{prop:trace-decay}), where $\mathbb E_a$ denotes the normalized average over $a\in\mathbb{S}^{n-1}$.
	To establish this decay, we decompose $\mathcal Q(a,\tau)$ along the compatible weight subspace: on the compatible subspace, Alessandrini's identity~\eqref{eq:alessandrini-intro} and the uniformly bounded partial inverses express $\mathcal Q(a,\tau)$ through correction terms whose trace norms vanish after averaging over $a$ (Lemmas~\ref{lem:nonlinear} and~\ref{lem:source-average}); on the excluded subspace of dimension $o(\tau^{n-2})$, the finite-rank inequality $\|\cdot\|_{\mathfrak S_1}\leq(\operatorname{rank})^{1/2}\|\cdot\|_{\mathfrak S_2}$ combines the $o(\tau^{n-2})$ rank bound from (i) with the $O(\tau^{-(n-2)/2})$ Hilbert--Schmidt decay of $\mathcal Q(a,\tau)$ so that the two factors cancel.
	
\end{enumerate}
Finally, a weighted energy identity for $q_1-q_2=0$ yields $\gamma_1=\gamma_2$, completing the proof of Theorem~\ref{thm:main} (Subsection~\ref{sec:recovery}).


\medskip
\subsection{Organization of the paper}\label{sec:organization}
The proof of Theorem~\ref{thm:main} proceeds across the remaining sections as follows.
Section~\ref{sec:reduction} reduces the equality of the Dirichlet-to-Neumann maps to the distributional Alessandrini identity at $W^{1,n}$ regularity, and introduces the free response operator $\mathcal Q(a,\tau)$ together with the averaged trace norm criterion for uniqueness (Proposition~\ref{prop:trace-decay}).
For each fixed direction $a\in\mathbb{S}^{n-1}$ and large frequency scale $\tau$, Section~\ref{sec:partial-averaging} constructs a uniformly bounded partial inverse for the CGO correction equation and shows that exact CGO solutions exist on a compatible weight subspace of codimension $o(\tau^{n-2})$.
Passing from fixed $a$ to averages over $a\in\mathbb{S}^{n-1}$ as $\tau\to\infty$, Section~\ref{sec:trace-method} uses these exact solutions to bound the trace norm of $\mathcal Q(a,\tau)$ on the compatible subspace and combines the $o(\tau^{n-2})$ codimension bound with an $O(\tau^{-(n-2)/2})$ Hilbert--Schmidt estimate on the excluded subspace to establish Proposition~\ref{prop:trace-decay}.
Finally, Section~\ref{sec:uniqueness} deduces $q_1=q_2$ from Proposition~\ref{prop:trace-decay} by recovering every Fourier frequency via small directional rotations, and completes the proof of Theorem~\ref{thm:main} through a weighted energy argument.
Appendix~\ref{app:auxiliary} supplies the low-frequency row estimate used in Section~\ref{sec:partial-averaging}.

\section{Reduction and the averaged trace norm criterion}\label{sec:reduction}
To prove uniqueness, we first turn equality of the Dirichlet-to-Neumann maps into an integral identity that can be tested with special solutions.
At $W^{1,n}$ regularity, this requires a common exterior extension and a weak interpretation of the associated Schr\"odinger potentials; the reduction is valid for $n\geq3$.
We then define an auxiliary operator for the harmonic leading terms of special solutions and state the averaged trace norm criterion for uniqueness when $n\geq4$.
\subsection{Reduction to an integral identity}\label{sec:boundary}\label{sec:potentials}
Let $n\geq3$ and let $\Omega\subset\mathbb R^n$ be a bounded Lipschitz domain.
Suppose that $\gamma_1,\gamma_2\in W^{1,n}(\Omega;\mathbb R)$ satisfy \eqref{eq:ellipticity} with common constants and that $\Lambda_{\gamma_1,\Omega}=\Lambda_{\gamma_2,\Omega}$.
We extend each conductivity to $\mathbb{R}^n$ so that its difference from $1$ is compactly supported and belongs to $W^{1,n}(\mathbb{R}^n)$, with the two extensions agreeing outside $\Omega$.
We then show that their Dirichlet-to-Neumann maps on a sufficiently large ball are equal.
Both extensions equal one near the boundary of this ball, so the Liouville transformation preserves the boundary data.
Equality of the boundary maps then yields the integral identity used below.

The common extension requires equality of the boundary traces on $\partial\Omega$.
Boundary determination from the Dirichlet-to-Neumann map has been studied since Kohn and Vogelius~\cite{kohn1984determining} and Brown~\cite{brown2001recovering}; for rough conductivities in $W^{s,p}(\Omega)$ ($1<p<\infty$, $s>1/p$), almost-everywhere trace recovery was established by Ponce-Vanegas~\cite[Theorem~1.1(A)]{ponce2020reconstruction}.
Applying this result with $s=1$ and $p=n$ yields the following lemma.
\begin{lemma}\label{lem:boundary}
	Let $n\geq3$ and let $\Omega\subset\mathbb{R}^n$ be a bounded Lipschitz domain.
	Suppose that $\gamma_1,\gamma_2\in W^{1,n}(\Omega;\mathbb{R})$ satisfy \eqref{eq:ellipticity}.
	If $\Lambda_{\gamma_1,\Omega}=\Lambda_{\gamma_2,\Omega}$, then $\operatorname{tr}\gamma_1=\operatorname{tr}\gamma_2$ almost everywhere on $\partial\Omega$, where $\operatorname{tr}$ denotes the Sobolev trace operator on $W^{1,n}(\Omega)$ (and on $H^1(\Omega)$ below).
\end{lemma}

We now combine the trace equality in Lemma~\ref{lem:boundary} with Sobolev extension.

\begin{lemma}\label{lem:gluing}
Let $n\geq3$ and let $\Omega\subset\mathbb{R}^n$ be a bounded Lipschitz domain.
Suppose that $\gamma_1,\gamma_2\in W^{1,n}(\Omega;\mathbb{R})$ satisfy \eqref{eq:ellipticity}.
If $\Lambda_{\gamma_1,\Omega}=\Lambda_{\gamma_2,\Omega}$, then there exist a ball $B$ with $\Omega\Subset B$ and real-valued extensions $\widetilde\gamma_j$ of $\gamma_j$ to $\mathbb{R}^n$ such that, for $j=1,2$,
\begin{equation}\label{eq:compact-extension}
\widetilde\gamma_j-1\in W^{1,n}(\mathbb{R}^n),\qquad
\operatorname{supp}(\widetilde\gamma_j-1)\Subset B.
\end{equation}
Moreover, $\min\{\gamma_{\min},1\}\leq\widetilde\gamma_j\leq\max\{\gamma_{\max},1\}$ almost everywhere in $\mathbb{R}^n$ for $j=1,2$.
The extensions agree almost everywhere outside $\Omega$ and satisfy
\begin{equation}\label{eq:extended-dtn}
\Lambda_{\widetilde\gamma_1,B}=\Lambda_{\widetilde\gamma_2,B}.
\end{equation}
\end{lemma}

\begin{proof}
We first construct extensions that agree outside $\Omega$ while preserving $W^{1,n}$ regularity and uniform ellipticity.
Set $h_j=\tfrac12\log\gamma_j$ in $\Omega$.
The uniform ellipticity bounds \eqref{eq:ellipticity} and the Sobolev chain rule give $h_j\in W^{1,n}(\Omega)\cap L^\infty(\Omega)$.
The Sobolev trace commutes with taking the logarithm on $[\gamma_{\min},\gamma_{\max}]$, so Lemma~\ref{lem:boundary} gives equal traces for $h_1$ and $h_2$.
Hence $h_1-h_2\in W^{1,n}_0(\Omega)$, and its zero extension $(h_1-h_2)_0$ belongs to $W^{1,n}(\mathbb{R}^n)$~\cite[Chapter~3]{ern2021finite}.

Extend $h_2$ to a function in $W^{1,n}(\mathbb{R}^n)$ by the Sobolev extension theorem~\cite[Chapter~2]{ern2021finite}, and truncate its values to the interval
\[
\left[\min\{0,\tfrac12\log\gamma_{\min}\},\ \max\{0,\tfrac12\log\gamma_{\max}\}\right].
\]
This preserves the Sobolev regularity and leaves $h_2$ unchanged in $\Omega$.
Since the interval contains $0$, multiplication by a smooth compactly supported cutoff with values in $[0,1]$, equal to one near $\overline\Omega$, preserves the same bounds.
Denote the resulting extension by $\widetilde h_2$ and set
\begin{equation}\label{eq:common-extension}
\widetilde h_1=\widetilde h_2+(h_1-h_2)_0.
\end{equation}
Then $\widetilde h_1=h_1$ in $\Omega$ and $\widetilde h_1=\widetilde h_2$ outside $\Omega$.
Thus $\widetilde\gamma_j=\exp(2\widetilde h_j)$ extends $\gamma_j$, the two extensions agree outside $\Omega$, and $\min\{\gamma_{\min},1\}\leq\widetilde\gamma_j\leq\max\{\gamma_{\max},1\}$.
The Sobolev chain rule gives $\widetilde\gamma_j-1\in W^{1,n}(\mathbb{R}^n)$, and both differences are compactly supported.
Choose a ball $B$ containing $\overline\Omega$ and both supports in its interior.
This proves \eqref{eq:compact-extension}.

To prove \eqref{eq:extended-dtn}, fix $\varphi\in H^{1/2}(\partial B)$ and let $u_1$ be the solution on $B$ for $\widetilde\gamma_1$ with trace $\varphi$.
Replace $u_1|_\Omega$ by the solution for $\gamma_2$ with the same trace $\psi=\operatorname{tr}(u_1|_\Omega)$ on $\partial\Omega$.
The matching traces give a glued function $u_2\in H^1(B)$ that still has trace $\varphi$ on $\partial B$.
Since the extensions and the functions agree outside $\Omega$, \eqref{eq:dn-intro} gives, for every $v\in H^1(B)$,
\[
\begin{aligned}
\int_B \widetilde\gamma_2\nabla u_2\cdot\nabla v\,\mathrm{d}x
-\int_B \widetilde\gamma_1\nabla u_1\cdot\nabla v\,\mathrm{d}x
&=\big\langle(\Lambda_{\gamma_2,\Omega}-\Lambda_{\gamma_1,\Omega})\psi,
   \operatorname{tr}(v|_\Omega)\big\rangle=0.
\end{aligned}
\]
For $v\in H^1_0(B)$ this shows that $u_2$ solves the equation for $\widetilde\gamma_2$.
For general $v\in H^1(B)$ it gives equality of the outer boundary pairings and hence \eqref{eq:extended-dtn}.
This data transfer uses only boundedness and uniform ellipticity on Lipschitz domains.
See also~\cite[Lemma~V.2]{knudsen2004reconstruction}.
\end{proof}

By Lemma~\ref{lem:gluing}, we may replace the conductivities by their extensions and omit the tildes.
Set $h_j=\tfrac12\log\gamma_j$, and choose a compact set $E\Subset B$ containing the supports of $h_1$ and $h_2$.
Then, for $j=1,2$,
\begin{equation}\label{eq:compact-model}
\begin{gathered}
h_j\in W^{1,n}(\mathbb{R}^n;\mathbb{R})\cap L^\infty(\mathbb{R}^n),
\qquad \operatorname{supp}h_j\subset E,\\[2pt]
\Lambda_{\exp(2h_1),B}=\Lambda_{\exp(2h_2),B}.
\end{gathered}
\end{equation}

As in \eqref{eq:schrodinger-intro}, we use the Liouville substitution $w_j=\exp(h_j)u_j$ for the conductivity equation with coefficient $\gamma_j=\exp(2h_j)$.

All function spaces below are complexified unless real values are specified.
Distributional and boundary duality pairings are bilinear, with no implicit conjugation.
The subscript $E$ denotes distributional support in $E$, and integrals without an indicated domain are over $\mathbb{R}^n$.

\begin{definition}\label{def:schrodinger-potentials}

For $h_j$ satisfying the regularity and support conditions in \eqref{eq:compact-model}, the Schr\"odinger potential associated with $\gamma_j=\exp(2h_j)$ is the compactly supported distribution

\[
q_j:=\Delta h_j+|\nabla h_j|^2,\qquad j=1,2,
\]
where $\Delta h_j$ is understood in the distributional sense.
The potentials $q_j$ and their difference $q_1-q_2$ have the form $q=\operatorname{div}F+V$ with $F\in L^n_E(\mathbb{R}^n;\mathbb{R}^n)$ and $V\in L^{n/2}_E(\mathbb{R}^n;\mathbb{R})$.
For $q_j$, we take $F=\nabla h_j$ and $V=|\nabla h_j|^2$.
For any such $q$, we define the product $qw$ with $w\in H^1_{\mathrm{loc}}$ by its pairing with $v\in H^1_{\mathrm{loc}}$,
\begin{equation}\label{eq:q-form}
\langle qw,v\rangle=\langle q,wv\rangle
:=-\int F\cdot(w\nabla v+v\nabla w)\,\mathrm{d}x
+\int Vwv\,\mathrm{d}x.
\end{equation}
\end{definition}

Since $F$ and $V$ are supported in $E\Subset B$, the form depends only on the values of $w,v$ near $E$ and is also well defined for $w,v\in H^1(B)$.
The Sobolev embedding on $B$ gives
\begin{equation}\label{eq:form-bound}
|\langle q,wv\rangle|
\leq C_{E,B}(\|F\|_n+\|V\|_{n/2})
\|w\|_{H^1(B)}\|v\|_{H^1(B)}.
\end{equation}
For smooth $w,v$, the form depends only on the distribution $q$, after inserting a smooth cutoff supported in $B$ and equal to one near $E$.
Density in $H^1(B)$ and estimate~\eqref{eq:form-bound} give the same independence of the representation $q=\operatorname{div}F+V$ for all $w,v\in H^1(B)$.
For fixed $w$, estimate~\eqref{eq:form-bound} shows that this pairing is a bounded linear functional of $v$, hence $qw\in H^{-1}_E$.
The exponents $n$ and $n/2$ also arise in unique continuation estimates for rough elliptic equations, as the integrability exponents for first- and zeroth-order coefficients, respectively~\cite{jerison1985unique,koch2001carleman}.

With this interpretation, let $h_1,h_2$ satisfy \eqref{eq:compact-model}, as obtained from the common exterior extensions in Lemma~\ref{lem:gluing}, and write their associated potentials as
\begin{equation}\label{eq:conductivity-potentials}
	q_j=\Delta h_j+|\nabla h_j|^2,
	\qquad j=1,2,
\end{equation}
where $\Delta h_j=\operatorname{div}(\nabla h_j)$ with $\nabla h_j\in L^n(B)$ and $|\nabla h_j|^2\in L^{n/2}(B)$ supported in $E$. Under \eqref{eq:compact-model}, the weak Liouville reduction and Alessandrini identity~\cite{alessandrini1988stable,haberman2015uniqueness} (see~\cite[Lemma~7.1]{haberman2015uniqueness} for the distributional formulation at $W^{1,n}$ regularity) yield
\begin{equation}\label{eq:alessandrini}
	\langle q_1-q_2,w_1w_2\rangle=0
\end{equation}
whenever $w_1,w_2\in H^1(B)$ satisfy $(-\Delta+q_j)w_j=0$ in the sense of distributions in $B$.

We next introduce the harmonic leading terms used to construct special solutions for \eqref{eq:alessandrini}.

\subsection{The free response operator and averaged trace norm decay}\label{sec:incident-response}\label{sec:trace-criterion}
The CGO argument recalled in Subsection~\ref{sec:methods} recovers the Fourier transform of the potential difference \(q_1-q_2\) from \eqref{eq:alessandrini} using exact solutions of the form \(w_\zeta=\exp(x\cdot\zeta)(1+\psi_\zeta)\), where \(\zeta\cdot\zeta=0\)~\cite{sylvester1986uniqueness,sylvester1987global,nachman1988reconstructions}.
The correction \(\psi_\zeta\) must solve \eqref{eq:cgo-correction-intro}.
For \(n\geq5\), the available estimates do not guarantee convergence of the usual iteration for \(\psi_\zeta\) under the \(W^{1,n}\) assumption alone~\cite{haberman2015uniqueness,ham2021uniqueness,ponce2021bilinear}.
We work with weighted superpositions of the exponentials
\[
\exp(\tau(a+\mathrm{i}b)\cdot x),
\qquad \tau\geq2,\quad |a|=|b|=1,\quad a\cdot b=0.
\]
They are harmonic because $(a+\mathrm{i}b)\cdot(a+\mathrm{i}b)=0$.

In Section~\ref{sec:partial-averaging}, we choose the weights so that these harmonic superpositions can be corrected to exact solutions of \eqref{eq:conductivity-potentials} for both potentials.

We first collect the pairings of the uncorrected leading terms in an auxiliary integral operator.

\begin{definition}\label{def:incident-response}
Fix $\tau\geq2$ and $a\in\mathbb{S}^{n-1}$.
Let $\sigma_a$ be normalized surface measure on $\mathbb{S}(a^\perp)=\{b\in\mathbb{S}^{n-1}:a\cdot b=0\}$.
For a weight function $g\in L^2(\mathbb{S}(a^\perp),\mathrm{d}\sigma_a)$, the \emph{Fourier extension operator} is defined by
\begin{equation}\label{eq:incident}
(\mathcal{E}_ag)(x)=\int_{\mathbb{S}(a^\perp)}\exp(\mathrm{i}\tau b\cdot x)g(b)\,\mathrm{d}\sigma_a(b).
\end{equation}
The function $\mathcal{E}_ag$ is the smooth oscillatory factor of the harmonic superposition
\begin{equation}\label{eq:superposition}
	\exp(\tau a\cdot x)(\mathcal{E}_ag)(x)
=\int_{\mathbb{S}(a^\perp)}
\exp(\tau(a+\mathrm{i}b)\cdot x)g(b)\,\mathrm{d}\sigma_a(b).
\end{equation}

For $g,h$ in this $L^2$ space, the real exponential factors in the harmonic functions $\exp(\tau a\cdot x)\mathcal{E}_ag$ and $\exp(-\tau a\cdot x)\overline{\mathcal{E}_ah}$ cancel in their product, leaving $(\mathcal{E}_ag)\overline{\mathcal{E}_ah}$.

The \emph{free response operator} $\mathcal{Q}(a,\tau)$ collects the pairings of the potential difference $q_1-q_2$ from Definition~\ref{def:schrodinger-potentials} with these products.
It is the bounded operator on the same $L^2$ space determined by
\begin{equation}\label{eq:Q-form}
(\mathcal{Q}(a,\tau)g,h)_{L^2(\mathbb{S}(a^\perp),\mathrm{d}\sigma_a)}
=\langle q_1-q_2,(\mathcal{E}_ag)\overline{\mathcal{E}_ah}\rangle.
\end{equation}
The inner product is linear in the first variable, and $q_1-q_2$ acts on the smooth product on the right.
We write simply $\mathcal{Q}$ when the parameters are fixed and suppress $\tau$ in $\mathcal{E}_a$.

\end{definition}

The term \emph{free} refers to the uncorrected harmonic superpositions in \eqref{eq:superposition}.
For fixed $a,\tau$, differentiating \eqref{eq:incident} and applying Cauchy--Schwarz bounds every fixed-order derivative of $\mathcal{E}_ag$ by a constant times $\|g\|_{L^2}$.

Since $q_1-q_2$ is compactly supported and of finite order, these derivative bounds give
\[
\bigl|\langle q_1-q_2,(\mathcal{E}_ag)\overline{\mathcal{E}_ah}\rangle\bigr|
\leq C_{q_1,q_2,a,\tau}\|g\|_{L^2}\|h\|_{L^2}.
\]
For each fixed $g$, the Riesz representation theorem therefore determines a unique vector $\mathcal{Q}(a,\tau)g$ in the same $L^2$ space satisfying \eqref{eq:Q-form}.
This vector depends linearly on $g$ and satisfies $\|\mathcal{Q}(a,\tau)g\|_{L^2}\leq C_{q_1,q_2,a,\tau}\|g\|_{L^2}$, which proves boundedness of the operator.
The pairing defining $\mathcal{Q}$ agrees with the weak pairing for $q_1-q_2$ in \eqref{eq:q-form}.

\begin{lemma}\label{lem:Q-kernel}
For fixed $a,\tau$ as in Definition~\ref{def:incident-response}, the free response operator has the integral representation
\[
(\mathcal{Q}(a,\tau)g)(b)
=\int_{\mathbb{S}(a^\perp)}K(b,c)g(c)\,\mathrm{d}\sigma_a(c).
\]
Here $c$ is the input direction and $b$ is the output direction.
With the unitary Fourier normalization $(2\pi)^{-n/2}$ and phase $\exp(-\mathrm{i}x\cdot\xi)$, the kernel $K=K_{a,\tau}$ is given by
\begin{equation}\label{eq:Q-kernel}
K(b,c)
=\langle q_1-q_2,\exp(\mathrm{i}\tau(c-b)\cdot x)\rangle
=(2\pi)^{n/2}\widehat{q_1-q_2}(\tau(b-c)).
\end{equation}
\end{lemma}

\begin{proof}
Substitute \eqref{eq:incident} into \eqref{eq:Q-form} and pass the pairing through the sphere integrals, which is justified by the compact support and finite order of $q_1-q_2$.
The exponential factors from $\mathcal{E}_ag$ at $c$ and $\overline{\mathcal{E}_ah}$ at $b$ multiply to $\exp(\mathrm{i}\tau(c-b)\cdot x)$, yielding the stated kernel by the Fourier convention.
\end{proof}

We use the Hilbert--Schmidt norm to estimate the kernel in \eqref{eq:Q-kernel}.
The trace norm controls the bounded operator tests used for Fourier recovery in Subsection~\ref{sec:frequencies}.

\begin{definition}\label{def:trace-ideals}

For a compact operator $\mathcal{S}$ between Hilbert spaces, its singular values $s_j(\mathcal{S})$ are the positive eigenvalues of $(\mathcal{S}^*\mathcal{S})^{1/2}$, counted with multiplicity in nonincreasing order and completed by zeros when there are only finitely many.

The trace norm and the Hilbert--Schmidt norm are given by
\[
\|\mathcal{S}\|_{\mathfrak S_1}=\sum_{j=1}^{\infty}s_j(\mathcal{S}),
\qquad
\|\mathcal{S}\|_{\mathfrak S_2}=\left(\sum_{j=1}^{\infty}s_j(\mathcal{S})^2\right)^{1/2}.
\]
The \emph{trace class} $\mathfrak S_1$ and the \emph{Hilbert--Schmidt class} $\mathfrak S_2$ consist of the compact operators for which the respective norms are finite.
When $\mathcal{S}:H\to H$ is trace class and $H$ is separable, its \emph{operator trace} is defined by
\[
\operatorname{Tr}\mathcal{S}:=\sum_j(\mathcal{S}e_j,e_j)_H,
\]
where $\{e_j\}$ is any orthonormal basis of $H$.
The series converges absolutely and its sum is independent of the choice of basis.
\end{definition}

For the free response operator, the Hilbert--Schmidt kernel formula gives
\[
\|\mathcal{Q}(a,\tau)\|_{\mathfrak S_2}^2
=\int_{\mathbb{S}(a^\perp)}\int_{\mathbb{S}(a^\perp)}
|K(b,c)|^2\,\mathrm{d}\sigma_a(c)\,\mathrm{d}\sigma_a(b).
\]
Thus \eqref{eq:Q-kernel} expresses this norm directly in terms of the Fourier transform of $q_1-q_2$.

The compact support of $q_1-q_2$ makes its Fourier transform $\widehat{q_1-q_2}$ smooth.
For fixed $a,\tau$, the kernel in \eqref{eq:Q-kernel} is therefore smooth on $\mathbb{S}(a^\perp)\times\mathbb{S}(a^\perp)$, so $\mathcal{Q}(a,\tau)$ is trace class by the smooth-kernel theorem~\cite[Corollary~4.4]{delgado2014schatten}.

We also use two standard inequalities to convert Hilbert--Schmidt bounds into trace norm bounds~\cite[Chapter~2]{simon2005trace}:
\[
\begin{aligned}
\|\mathcal{S}\mathcal{T}\|_{\mathfrak S_1}
\leq\|\mathcal{S}\|_{\mathfrak S_2}\|\mathcal{T}\|_{\mathfrak S_2},\qquad
\|\mathcal{S}\|_{\mathfrak S_1}
\leq\sqrt{\operatorname{rank}\mathcal{S}}\,\|\mathcal{S}\|_{\mathfrak S_2}.
\end{aligned}
\]
The first holds for composable Hilbert--Schmidt operators, and the second for finite-rank operators.
They will control the correction terms and the finite-rank remainder, respectively.

The construction of $\mathcal{Q}(a,\tau)$ uses the normalized surface measure $\sigma_a$ on the transverse sphere $\mathbb{S}(a^\perp)$ with $a$ fixed.
We now use the normalized surface measure $\mu$ on $\mathbb{S}^{n-1}$, defined below, to average $\|\mathcal{Q}(a,\tau)\|_{\mathfrak S_1}$ over $a$ at each fixed $\tau$, with the aim of proving that this average tends to zero as $\tau\to\infty$.

\begin{definition}\label{def:direction-average}
Let $\mu$ be the uniform probability measure on $\mathbb{S}^{n-1}$, obtained by dividing surface measure by the total surface area $|\mathbb{S}^{n-1}|$.
For fixed $\tau$ and a nonnegative measurable or $\mu$-integrable function $\Phi(\cdot,\tau)$, its \emph{spherical average} is the expectation of $\Phi(a,\tau)$ when $a$ has distribution $\mu$,
\begin{equation}\label{eq:direction-average}
\mathbb E_a \Phi(a,\tau):=\int_{\mathbb{S}^{n-1}}\Phi(a,\tau)\,\mathrm{d}\mu(a).
\end{equation}
\end{definition}

For each fixed $\tau$, the function $a\mapsto\|\mathcal{Q}(a,\tau)\|_{\mathfrak S_1}$ is measurable; see Subsection~\ref{sec:trace}.
If its average tends to zero as $\tau\to\infty$, Proposition~\ref{prop:frequency} shows that $\widehat{q_1-q_2}$ vanishes at every frequency, hence $q_1=q_2$.
The weighted energy argument in Subsection~\ref{sec:recovery} then gives equality of the conductivities.

It therefore remains to prove that equality of the Dirichlet-to-Neumann maps implies this averaged trace norm decay.
This is the content of the following proposition, proved in Subsection~\ref{sec:trace}.

\begin{proposition}\label{prop:trace-decay}
Let $n\geq4$ and let $h_1,h_2$ satisfy \eqref{eq:compact-model}.
Let $q_j$ be the associated potentials in Definition~\ref{def:schrodinger-potentials}.
Then the free response operator $\mathcal{Q}(a,\tau)$ defined by \eqref{eq:Q-form} satisfies
\begin{equation}\label{eq:Ea}
	\mathbb E_a\|\mathcal{Q}(a,\tau)\|_{\mathfrak S_1}\longrightarrow0
\qquad\text{as }\tau\longrightarrow\infty.
\end{equation}
\end{proposition}

\section{Partial inversion and exact solutions}\label{sec:partial-averaging}
Section~\ref{sec:reduction} derives the integral identity~\eqref{eq:alessandrini} and states the averaged trace norm decay needed for uniqueness in Proposition~\ref{prop:trace-decay}.
This section constructs a partial inverse, bounds its exceptional dimension, and produces the compatible exact solutions used to prove that decay.

\subsection{A bounded formulation of the conjugated equation}\label{sec:operators}
The integral identity~\eqref{eq:alessandrini} requires an exact solution of the Schr\"odinger equation~\eqref{eq:conductivity-potentials} for each $q_j$.
Fix $j\in\{1,2\}$, put $q=q_j$, and choose $\tau\geq2$, $a\in\mathbb{S}^{n-1}$, and a weight $g\in L^2(\mathbb{S}(a^\perp),\mathrm{d}\sigma_a)$ as in Definition~\ref{def:incident-response}.
The function $\exp(\tau a\cdot x)\mathcal{E}_ag$ in \eqref{eq:superposition} is harmonic, but for this $q$ it need not solve the Schr\"odinger equation~\eqref{eq:conductivity-potentials}.
We seek $r\in H^1(B)$ such that $w=\exp(\tau a\cdot x)(\mathcal{E}_ag+r)$ solves $(-\Delta+q)w=0$ in $B$; thus $r$ is the correction to the oscillatory factor $\mathcal{E}_ag$.

With $\widetilde w=\mathcal{E}_ag+r$, factoring out the real exponential in \eqref{eq:conductivity-potentials} gives the exponentially conjugated Schr\"odinger equation
\begin{equation}\label{eq:conjugated-equation}
(-\Delta-2\tau a\cdot\nabla-\tau^2+q)\widetilde w=0,
\end{equation}
where multiplication by $q$ is defined by \eqref{eq:q-form}; for exponential conjugation in CGO constructions, see~\cite{haberman2015uniqueness, kenig2013calderon}.

To derive the equation for $r$, write the free part of \eqref{eq:conjugated-equation} as a Fourier multiplier.
With $D=-\mathrm{i}\nabla$, its symbol is
\begin{equation}\label{eq:symbol-intro}
d_a(\xi)=|\xi|^2-\tau^2-2\mathrm{i}\tau a\cdot\xi.
\end{equation}
Thus $d_a(D)=-\Delta-2\tau a\cdot\nabla-\tau^2$.
Each mode in \eqref{eq:incident} has frequency $\tau b$ for $b\in\mathbb{S}(a^\perp)$, so $d_a(\tau b)=0$ and hence $d_a(D)\mathcal{E}_ag=0$.
Substituting $\mathcal{E}_ag+r$ into \eqref{eq:conjugated-equation} therefore gives, in the sense of distributions on $B$,
\begin{equation}\label{eq:conjugated-correction}
(d_a(D)+q)r=-q\mathcal{E}_ag.
\end{equation}
The weak multiplication in Definition~\ref{def:schrodinger-potentials} applies to both $\mathcal{E}_ag$ and the sought $r$ since they belong to $H^1(B)$.
The bound~\eqref{eq:form-bound} gives $q\mathcal{E}_ag\in H^{-1}_E$.
In the estimates below, constants may depend on the fixed dimension, supports, cutoffs, and coefficients, but are uniform in $a$ and all sufficiently large $\tau$ unless stated otherwise.

For $n\geq5$, the available CGO estimates do not justify the usual iteration for $r$ under the $W^{1,n}$ assumption alone~\cite{haberman2015uniqueness,ham2021uniqueness,ponce2021bilinear}.
Our central new step toward solving \eqref{eq:conjugated-correction} is a quantitative partial inverse for its bounded $L^2$ realization (Proposition~\ref{prop:partial}): it has a norm bound uniform in $a$ and large $\tau$, while the excluded singular subspaces have dimension $o(\tau^{n-2})$.
The correction equation~\eqref{eq:conjugated-correction} is not directly an equation on $L^2$: the correction is sought in $H^1(B)$, while its potential term is an $H^{-1}_E$ distribution.
To use spectral projections on $L^2$, we seek a correction of the form $r=\mathcal{A}z$ with $z\in L^2$, where
\[
\mathcal{A}=|d_a(D)|^{-1/2}.
\]
Wherever $d_a\ne0$, the symbol product $|d_a|^{-1/2}d_a|d_a|^{-1/2}$ equals $d_a/|d_a|$, which has modulus one.
The next lemma shows that $\mathcal{A}$ also maps $L^2$ to local $H^1$ and $H^{-1}_E$ to $L^2$, so it can be placed on both sides of the potential term.

\begin{lemma}\label{lem:A-maps}
Let $U\subset\mathbb{R}^n$ be a bounded open set and $E\subset\mathbb{R}^n$ a compact set.
For every $\tau\geq2$ and $a\in\mathbb{S}^{n-1}$, the multiplier $\mathcal{A}=|d_a(D)|^{-1/2}$ defines bounded operators
\begin{equation}\label{eq:A-maps}
\mathcal{A}:L^2(\mathbb{R}^n)\longrightarrow H^1(U),
\qquad
\mathcal{A}:H^{-1}_E(\mathbb{R}^n)\longrightarrow L^2(\mathbb{R}^n).
\end{equation}
Both operator norms are at most $C_{n,U,E}\tau^{1/2}$, where the constant is independent of $a$ and $\tau$.
\end{lemma}
\begin{proof}
After a Fourier translation, the localized estimates in \cite[Lemma~2.2]{haberman2013uniqueness} give the first map in \eqref{eq:A-maps} with the stated $\tau^{1/2}$ bound.
The second map follows by duality because its inputs are supported in $E$.
\end{proof}

For a real potential $q=\operatorname{div}F+V$ with $F\in L^n_E(\mathbb{R}^n;\mathbb{R}^n)$ and $V\in L^{n/2}_E(\mathbb{R}^n;\mathbb{R})$, the form estimate~\eqref{eq:form-bound} gives $q:H^1_{\mathrm{loc}}\to H^{-1}_E$.
By Lemma~\ref{lem:A-maps} and \eqref{eq:form-bound}, the map $z\mapsto\mathcal{A}(q\mathcal{A}z)$ is bounded on $L^2$: $q\mathcal{A}z\in H^{-1}_E$ is defined using the local $H^1$ regularity of $\mathcal{A}z$ near $E$.
For the fixed $a$ and $\tau$, define the free operator $\mathcal{T}_0$ and the bounded operator $\mathcal{T}_q$ on $L^2(\mathbb{R}^n)$ by
\begin{equation}\label{eq:T-definition}
\mathcal{T}_0:=\frac{d_a(D)}{|d_a(D)|},\qquad \mathcal{T}_q:=\mathcal{T}_0+\mathcal{A}q\mathcal{A}.
\end{equation}
The quotient has modulus one away from the measure-zero set where $d_a=0$; we set it equal to $1$ on that set.
Thus $\mathcal{T}_0$ is unitary.
The dependence of $\mathcal{A}$ and $\mathcal{T}_q$ on $a,\tau$ is suppressed when these parameters are fixed; the same convention applies to the partial inverse, source operators, and exceptional projections introduced below.
The next lemma explains how the bounded $L^2$ equation yields a distributional solution of the correction equation.

\begin{lemma}\label{lem:bounded-correction}
Fix $\tau\geq2$ and $a\in\mathbb{S}^{n-1}$, and let $E\subset\mathbb{R}^n$ be compact.
Let $q=\operatorname{div}F+V$, where $F\in L^n_E(\mathbb{R}^n;\mathbb{R}^n)$ and $V\in L^{n/2}_E(\mathbb{R}^n;\mathbb{R})$.
For $z\in L^2(\mathbb{R}^n)$ and $s\in H^{-1}_E(\mathbb{R}^n)$, the following identities hold in the sense of distributions:
\begin{equation}\label{eq:A-identities}
d_a(D)\mathcal{A}z=|d_a(D)|^{1/2}\mathcal{T}_0z,
\qquad |d_a(D)|^{1/2}\mathcal{A}s=s.
\end{equation}
Consequently, $(d_a(D)+q)\mathcal{A}z=|d_a(D)|^{1/2}\mathcal{T}_qz$.
In particular, if $\mathcal{T}_qz=\mathcal{A}s$, then $r=\mathcal{A}z$ satisfies
\begin{equation}\label{eq:exact-from-T}
(d_a(D)+q)\mathcal{A}z=s
\end{equation}
in the sense of distributions.
\end{lemma}

\begin{proof}
The bound $|d_a(\xi)|^{1/2}\leq|\xi|+\tau$ makes $|d_a(D)|^{1/2}$ a map from $L^2$ to $H^{-1}$.
Away from the measure-zero set where $d_a=0$, the Fourier symbols satisfy
\[
d_a|d_a|^{-1/2}=|d_a|^{1/2}\frac{d_a}{|d_a|},\qquad |d_a|^{1/2}|d_a|^{-1/2}=1.
\]
These equalities give \eqref{eq:A-identities} in the sense of distributions.
Since $q\mathcal{A}z\in H^{-1}_E$ by Lemma~\ref{lem:A-maps} and \eqref{eq:form-bound}, \eqref{eq:T-definition} and \eqref{eq:A-identities} give $(d_a(D)+q)\mathcal{A}z=|d_a(D)|^{1/2}\mathcal{T}_qz$.
If $\mathcal{T}_qz=\mathcal{A}s$, the second identity in \eqref{eq:A-identities} gives $|d_a(D)|^{1/2}\mathcal{T}_qz=|d_a(D)|^{1/2}\mathcal{A}s=s$, proving \eqref{eq:exact-from-T}.
\end{proof}

Taking $s=-q\mathcal{E}_ag$ in Lemma~\ref{lem:bounded-correction} recovers the correction equation~\eqref{eq:conjugated-correction}.
To transfer pairings between the distributional and $L^2$ formulations, we also use
\begin{equation}\label{eq:A-duality}
\langle s,\overline{\mathcal{A}z}\rangle=(\mathcal{A}s,z)_{L^2},
\qquad s\in H^{-1}_E, z\in L^2.
\end{equation}
For $z$ with Fourier support away from $\{d_a=0\}$, this is Fourier duality; Lemma~\ref{lem:A-maps} extends it to every $z\in L^2$ by density.

\subsection{A partial inverse with controlled exceptional dimension}\label{sec:partial-inverse}
By \eqref{eq:exact-from-T}, it suffices to solve $\mathcal{T}_qz=-\mathcal{A}(q\mathcal{E}_ag)$ and set $r=\mathcal{A}z$ to obtain a solution of \eqref{eq:conjugated-correction}.
We seek a partial inverse of $\mathcal{T}_q$: it solves this $L^2$ equation for right-hand sides orthogonal to a finite-dimensional excluded space and returns the unique solution orthogonal to a second excluded space.
We will bound this inverse uniformly for large $\tau$ and show that both excluded dimensions are $o(\tau^{n-2})$.
This bound, uniform in $a$ and $\tau$, allows us to deduce averaged trace norm decay of the correction terms from the source estimates in Section~\ref{sec:trace-method}.

For the estimates below, write $\mathcal{K}_q=\mathcal{A}q\mathcal{A}$, so that $\mathcal{T}_q=\mathcal{T}_0+\mathcal{K}_q$.
Taking $s=q\mathcal{A}z_1$ and $z=z_2$ in \eqref{eq:A-duality} gives
\begin{equation}\label{eq:K-form}
(\mathcal{K}_qz_1,z_2)_{L^2}=\langle q,(\mathcal{A}z_1)\overline{\mathcal{A}z_2}\rangle.
\end{equation}
Since $q$ is real, $\mathcal{K}_q$ is self-adjoint.

Fix $R>2$.
Let $\mathcal{L}_\tau:L^2(\mathbb{R}^n)\to L^2(\mathbb{R}^n)$ be the Fourier multiplier with symbol $\mathbf{1}_{\{|\xi|\leq R\tau\}}$, that is,
\[
\widehat{\mathcal{L}_\tau z}(\xi)
=\mathbf{1}_{\{|\xi|\leq R\tau\}}\widehat z(\xi),
\qquad z\in L^2(\mathbb{R}^n).
\]
Set $\mathcal{H}_\tau=I-\mathcal{L}_\tau$, the complementary Fourier projection onto $|\xi|>R\tau$.
The cutoff $R$ remains fixed in this construction and its later applications; its dependence is suppressed in $\mathcal{L}_\tau$, $\mathcal{H}_\tau$, the threshold, exceptional projections, and partial inverse introduced below.
We call $\mathcal{L}_\tau\mathcal{K}_q$ the \emph{full low-frequency row} of $\mathcal{K}_q$.
Here $\mathcal{L}_\tau$ restricts the output to frequencies $|\xi|\leq R\tau$, while the input ranges over all of $L^2$.
We call $\mathcal{H}_\tau\mathcal{K}_q\mathcal{H}_\tau$ the \emph{high-frequency block} of $\mathcal{K}_q$, with both input and output restricted to frequencies $|\xi|>R\tau$.
We first bound the full low-frequency row in the Hilbert--Schmidt norm $\|\cdot\|_{\mathfrak S_2}$ (Definition~\ref{def:trace-ideals}).
Its squared norm will control the number of exceptional directions.

The next lemma introduces the potential-dependent factor in the low-frequency estimate and proves its decay.
\begin{lemma}\label{lem:gamma-decay}
Let $n>2$ and let $q\in H^{-1}(\mathbb{R}^n)$ have compact support.
For $\tau\geq2$, define
\begin{equation}\label{eq:gamma-intro}
\Gamma_q(\tau)=\int_{\mathbb{R}^n}
\frac{|\widehat q(\xi)|^2}{|\xi|^2+\tau|\xi|}\,\mathrm{d}\xi.
\end{equation}
The homogeneous norm used here is
\begin{equation}\label{eq:homogeneous-norm}
\|q\|_{\dot H^{-1}}^2
=\int\frac{|\widehat q(\xi)|^2}{|\xi|^2}\,\mathrm{d}\xi.
\end{equation}
Then $\Gamma_q(\tau)$ is independent of $a$, satisfies $0\leq\Gamma_q(\tau)\leq\|q\|_{\dot H^{-1}}^2<\infty$, and tends to zero as $\tau\to\infty$.
\end{lemma}

\begin{proof}
The high-frequency part of \eqref{eq:homogeneous-norm} is controlled by $\|q\|_{H^{-1}}^2$.
At low frequencies, $\widehat q$ is bounded by compact support and $|\xi|^{-2}$ is locally integrable for $n>2$, so the norm is finite.
Since $(|\xi|^2+\tau|\xi|)^{-1}\leq|\xi|^{-2}$ and tends to zero for $\xi\ne0$, comparison gives the bound and dominated convergence gives $\Gamma_q(\tau)\to0$.
\end{proof}

For $n\geq4$, the assumptions on $F,V$ and compact support give $F,V\in L^2$, hence $q\in H^{-1}_E$ and Lemma~\ref{lem:gamma-decay} applies.
The next proposition gives the promised Hilbert--Schmidt estimate for the full low-frequency row.

\begin{proposition}\label{prop:row}
Let $n\geq4$ and $R>2$, and let $q=\operatorname{div}F+V$ be real, with $F\in L^n_E$ and $V\in L^{n/2}_E$ for a compact set $E$.
Then, for every $\tau\geq2$,
\begin{equation}\label{eq:row}
\sup_{a\in \mathbb{S}^{n-1}}\|\mathcal{L}_\tau \mathcal{K}_q\|_{\mathfrak S_2}^2
\leq C_{n,R}\tau^{n-2}\Gamma_q(\tau).
\end{equation}
The constant $C_{n,R}$ depends only on $n$ and $R$; in particular, it is independent of $q$, its support $E$, $a$, and $\tau$.
\end{proposition}

The proof in Appendix~\ref{app:auxiliary} uses the zeros of the symbol $d_a$.
After scaling to $\tau=1$, they form a sphere of codimension two.
An intersection bound for thin neighborhoods of this sphere and its translates controls the convolution of two inverse-symbol factors; the Fourier kernel of $\mathcal{L}_\tau\mathcal{K}_q$ then gives \eqref{eq:row}.
By Lemma~\ref{lem:gamma-decay}, the right-hand side of \eqref{eq:row} is $o(\tau^{n-2})$, uniformly in $a$.
The next proposition constructs a partial inverse of $\mathcal{T}_q$.

\begin{proposition}\label{prop:partial}
Fix $n\geq4$ and $R>2$, and let $q$ be a real potential as in Proposition~\ref{prop:row}.
There is a threshold $\tau_0(q)\geq2$, independent of $a$, such that for each $a\in\mathbb{S}^{n-1}$ and $\tau\geq\tau_0(q)$ there are finite-rank orthogonal projections $\Pi_q^{\mathrm{in}},\Pi_q^{\mathrm{out}}$ and a bounded operator $\mathcal{W}_q$ on $L^2$.
The projections satisfy
\begin{equation}\label{eq:rank}
\max\{\operatorname{rank}\Pi_q^{\mathrm{out}},\operatorname{rank}\Pi_q^{\mathrm{in}}\}
\leq C_{n,R}\tau^{n-2}\Gamma_q(\tau),
\end{equation}
where $C_{n,R}$ depends only on $n$ and $R$.
The partial inverse $\mathcal{W}_q$ satisfies
\begin{equation}\label{eq:partial-identities}
\mathcal{T}_q\mathcal{W}_q=I-\Pi_q^{\mathrm{out}},
\qquad \mathcal{W}_q\mathcal{T}_q=I-\Pi_q^{\mathrm{in}},
\qquad\|\mathcal{W}_q\|\leq4.
\end{equation}
In particular, for every $y\in\ker\Pi_q^{\mathrm{out}}$, the vector $z=\mathcal{W}_qy\in\ker\Pi_q^{\mathrm{in}}$ is the unique solution of $\mathcal{T}_qz=y$ in $\ker\Pi_q^{\mathrm{in}}$.
\end{proposition}

\begin{remark}\label{rem:partial-inverse-meaning}
The operator $\mathcal{W}_q$ is a \emph{partial inverse} of $\mathcal{T}_q$ in the sense that it inverts $\mathcal{T}_q$ between the retained input and output subspaces.
The projections $\Pi_q^{\mathrm{in}}$ and $\Pi_q^{\mathrm{out}}$ select the input and output directions of $\mathcal{T}_q$ that are excluded from the inversion, respectively.
In the proof, these are chosen to correspond to singular values below $1/4$.
The complementary projections $I-\Pi_q^{\mathrm{in}}$ and $I-\Pi_q^{\mathrm{out}}$ retain the remaining directions.
On these subspaces,
\[
\mathcal{T}_q:\ker\Pi_q^{\mathrm{in}}\longrightarrow\ker\Pi_q^{\mathrm{out}}
\]
is invertible.
The operator $\mathcal{W}_q$ is the inverse of this restriction, extended by zero on $\operatorname{ran}\Pi_q^{\mathrm{out}}$.

The high-frequency estimate gives a uniformly invertible comparison operator, and the Hilbert--Schmidt bound for the remainder controls the number of singular directions of $\mathcal{T}_q$ with singular values below $1/4$.
Inverting on the complementary subspaces gives the uniform bound needed in Section~\ref{sec:trace-method}.
The compatibility conditions in Subsection~\ref{sec:sources} then eliminate the residual in the excluded output subspace.

For a general right-hand side $y\in L^2$, \eqref{eq:partial-identities} gives
\[
\mathcal{T}_q(\mathcal{W}_qy)=y-\Pi_q^{\mathrm{out}}y.
\]
Thus $\mathcal{W}_qy$ solves $\mathcal{T}_qz=y$ exactly when $\Pi_q^{\mathrm{out}}y=0$; otherwise the excluded component $\Pi_q^{\mathrm{out}}y$ remains as a residual.
\end{remark}

\begin{proof}[Proof of Proposition~\ref{prop:partial}]

\medskip
\noindent\textit{Step 1.
Smallness of the high-frequency block.}
We show that the high-frequency block $\mathcal{H}_\tau\mathcal{K}_q\mathcal{H}_\tau$ of $\mathcal{K}_q$ tends to zero in operator norm as $\tau\to\infty$, uniformly in $a$.
By \eqref{eq:symbol-intro}, the real part of $d_a(\xi)$ is positive when $|\xi|>R\tau$.
Since then $\tau^2<R^{-2}|\xi|^2$,
\begin{equation}\label{eq:elliptic-symbol}
|d_a(\xi)|\geq\operatorname{Re}d_a(\xi)=|\xi|^2-\tau^2\geq(1-R^{-2})|\xi|^2.
\end{equation}
On the range of $\mathcal{H}_\tau$, \eqref{eq:elliptic-symbol} gives stronger global bounds for $\mathcal{A}$ than the local bound in \eqref{eq:A-maps}.
Indeed, Plancherel's theorem (the $L^2$-isometry of the Fourier transform; see \cite[Ch.~6, \S2]{stein2003fourier}) gives, for every $z\in L^2(\mathbb{R}^n)$,
\[
\begin{aligned}
\|\mathcal{A}\mathcal{H}_\tau z\|_2^2
&=\int_{|\xi|>R\tau}\frac{|\widehat z(\xi)|^2}{|d_a(\xi)|}\,\mathrm{d}\xi\leq C_R\tau^{-2}\|z\|_2^2,\\
\|\nabla(\mathcal{A}\mathcal{H}_\tau z)\|_2^2
&=\int_{|\xi|>R\tau}\frac{|\xi|^2|\widehat z(\xi)|^2}{|d_a(\xi)|}\,\mathrm{d}\xi\leq C_R\|z\|_2^2.
\end{aligned}
\]
Here the first bound also uses $|\xi|>R\tau$, while the second uses $|\xi|^2/|d_a(\xi)|\leq(1-R^{-2})^{-1}$.
The homogeneous Sobolev inequality then gives $\|\mathcal{A}\mathcal{H}_\tau z\|_{2^*}\leq C_{n,R}\|z\|_2$, where $2^*=2n/(n-2)$.
Together these estimates read
\begin{equation}\label{eq:high-maps}
\tau\|\mathcal{A}\mathcal{H}_\tau z\|_2+\|\nabla(\mathcal{A}\mathcal{H}_\tau z)\|_2+\|\mathcal{A}\mathcal{H}_\tau z\|_{2^*}\leq C_{n,R}\|z\|_2.
\end{equation}
We now estimate the high-frequency block through its quadratic form.
Because $\mathcal{H}_\tau$ is an orthogonal projection, \eqref{eq:K-form} gives
\[
(\mathcal{H}_\tau\mathcal{K}_q\mathcal{H}_\tau z,z)
=\langle q,|\mathcal{A}\mathcal{H}_\tau z|^2\rangle,
\qquad z\in L^2.
\]
Applying the potential form \eqref{eq:q-form} from Definition~\ref{def:schrodinger-potentials} and the product rule, we obtain
\[
\big|(\mathcal{H}_\tau\mathcal{K}_q\mathcal{H}_\tau z,z)\big|\leq 2\int |F|\,|\mathcal{A}\mathcal{H}_\tau z|\,|\nabla(\mathcal{A}\mathcal{H}_\tau z)|\,\mathrm{d}x+\int |V|\,|\mathcal{A}\mathcal{H}_\tau z|^2\,\mathrm{d}x.
\]
To make these integrals small without assuming $F$ or $V$ is bounded, split each coefficient into a bounded part and a tail that is small in $L^n$ or $L^{n/2}$.
For a truncation level $\alpha>0$, put $F_\alpha=F\mathbf{1}_{\{|F|\leq\alpha\}}$ and $V_\alpha=V\mathbf{1}_{\{|V|\leq\alpha\}}$.
Since $F\in L^n$ and $V\in L^{n/2}$, the tails satisfy $\|F-F_\alpha\|_n+\|V-V_\alpha\|_{n/2}\to0$ as $\alpha\to\infty$, while $F_\alpha$ and $V_\alpha$ are bounded for fixed $\alpha$.
H\"older's inequality and \eqref{eq:high-maps} control the $F-F_\alpha$ term by the $L^{2^*}$ and gradient bounds, and the $V-V_\alpha$ term by two $L^{2^*}$ bounds.
For $F_\alpha$, one $L^2$ bound and one gradient bound give a factor $\tau^{-1}$; for $V_\alpha$, two $L^2$ bounds give $\tau^{-2}$.
Since $q$ is real, $\mathcal{H}_\tau\mathcal{K}_q\mathcal{H}_\tau$ is bounded and self-adjoint.
The quadratic-form norm formula for bounded self-adjoint operators~\cite[Chapter~VI, Problem~9(a)]{reed1980functional} identifies its norm with the supremum of the absolute value of this form over $\|z\|_2=1$.
Consequently,
\begin{equation}\label{eq:high-small}
\|\mathcal{H}_\tau \mathcal{K}_q\mathcal{H}_\tau\|\leq C_{n,R}\bigl(\|F-F_\alpha\|_n+\|V-V_\alpha\|_{n/2}+\tau^{-1}\|F_\alpha\|_\infty+\tau^{-2}\|V_\alpha\|_\infty\bigr).
\end{equation}
First choose $\alpha$ so that the tail terms in \eqref{eq:high-small} are small; with this $\alpha$ fixed, let $\tau\to\infty$ to make the bounded terms vanish.
All bounds are independent of $a$, so
\begin{equation}\label{eq:high-limit}
\sup_a\|\mathcal{H}_\tau \mathcal{K}_q\mathcal{H}_\tau\|\longrightarrow0.
\end{equation}

\medskip
\noindent\textit{Step 2.
The comparison operator and the remainder.}
By \eqref{eq:high-limit}, choose $\tau_0(q)\geq2$ so that $\|\mathcal{H}_\tau\mathcal{K}_q\mathcal{H}_\tau\|\leq1/2$ for every $a\in\mathbb{S}^{n-1}$ and $\tau\geq\tau_0(q)$.
Define the comparison operator
\[
\mathcal{T}_{\mathrm{ref}}=\mathcal{T}_0+\mathcal{H}_\tau \mathcal{K}_q\mathcal{H}_\tau=\mathcal{T}_0(I+\mathcal{T}_0^{-1}\mathcal{H}_\tau \mathcal{K}_q\mathcal{H}_\tau).
\]
Since $\mathcal{T}_0$ is unitary, $\|\mathcal{T}_0^{-1}\mathcal{H}_\tau\mathcal{K}_q\mathcal{H}_\tau\|\leq1/2$.
The Neumann-series criterion~\cite[Section~VI.3]{reed1980functional} makes $I+\mathcal{T}_0^{-1}\mathcal{H}_\tau\mathcal{K}_q\mathcal{H}_\tau$, and hence $\mathcal{T}_{\mathrm{ref}}$, invertible with $\|\mathcal{T}_{\mathrm{ref}}^{-1}\|\leq2$.
Thus both $\mathcal{T}_{\mathrm{ref}}$ and $\mathcal{T}_{\mathrm{ref}}^*$ are bounded below by $1/2$.
The remainder is $\mathcal{R}_q=\mathcal{T}_q-\mathcal{T}_{\mathrm{ref}}=\mathcal{K}_q-\mathcal{H}_\tau\mathcal{K}_q\mathcal{H}_\tau$.
Since $I=\mathcal{L}_\tau+\mathcal{H}_\tau$, this is
\[
\mathcal{R}_q=\mathcal{L}_\tau\mathcal{K}_q\mathcal{L}_\tau+\mathcal{L}_\tau\mathcal{K}_q\mathcal{H}_\tau+\mathcal{H}_\tau\mathcal{K}_q\mathcal{L}_\tau.
\]
Proposition~\ref{prop:row} shows that $\mathcal{L}_\tau\mathcal{K}_q$ is Hilbert--Schmidt.
By the Hilbert--Schmidt ideal property~\cite[Section~VI.6]{reed1980functional}, so are $\mathcal{L}_\tau\mathcal{K}_q\mathcal{L}_\tau$ and $\mathcal{L}_\tau\mathcal{K}_q\mathcal{H}_\tau$.
Since $\mathcal{K}_q$ is self-adjoint, the third block is the adjoint of the second and has the same Hilbert--Schmidt norm.
The three blocks are pairwise orthogonal in the Hilbert--Schmidt inner product because each pair has orthogonal input or output spaces.
Hence
\begin{equation}\label{eq:remainder-row}
\|\mathcal{R}_q\|_{\mathfrak S_2}^2
=\|\mathcal{L}_\tau \mathcal{K}_q\mathcal{L}_\tau\|_{\mathfrak S_2}^2
+2\|\mathcal{L}_\tau \mathcal{K}_q\mathcal{H}_\tau\|_{\mathfrak S_2}^2
\leq2\|\mathcal{L}_\tau \mathcal{K}_q\|_{\mathfrak S_2}^2.
\end{equation}

\medskip
\noindent\textit{Step 3.
Rank bounds for the exceptional projections.}
To exclude directions on which $\mathcal{T}_q$ is too small for a uniformly bounded inverse, define the input and output spectral projections by
\begin{equation}\label{eq:bad-projections}
\Pi_q^{\mathrm{in}}=\mathbf{1}_{[0,1/16)}(\mathcal{T}_q^*\mathcal{T}_q),
\qquad\Pi_q^{\mathrm{out}}=\mathbf{1}_{[0,1/16)}(\mathcal{T}_q\mathcal{T}_q^*).
\end{equation}
By the spectral theorem for bounded self-adjoint operators in its bounded Borel functional calculus form~\cite[Sections~VII.2--VII.3]{reed1980functional}, the two expressions in \eqref{eq:bad-projections} define orthogonal spectral projections for the positive operators $\mathcal{T}_q^*\mathcal{T}_q$ and $\mathcal{T}_q\mathcal{T}_q^*$.
For $z\in\operatorname{ran}\Pi_q^{\mathrm{in}}$, spectral calculus and $1/16=(1/4)^2$ give $\|\mathcal{T}_qz\|\leq\|z\|/4$.
The lower bound for $\mathcal{T}_{\mathrm{ref}}$ from Step~2 therefore gives
\[
\|\mathcal{R}_qz\|\geq\|\mathcal{T}_{\mathrm{ref}}z\|-\|\mathcal{T}_qz\|\geq\|z\|/4.
\]

For any orthonormal family $z_1,\ldots,z_N$ in $\operatorname{ran}\Pi_q^{\mathrm{in}}$, the Hilbert--Schmidt norm formula and Bessel's inequality give
\[
\begin{aligned}
N
\leq 16\sum_{\ell=1}^N\|\mathcal{R}_qz_\ell\|^2
\leq 16\|\mathcal{R}_q\|_{\mathfrak S_2}^2
\leq 32\|\mathcal{L}_\tau\mathcal{K}_q\|_{\mathfrak S_2}^2
\leq 32C_{n,R}\tau^{n-2}\Gamma_q(\tau).
\end{aligned}
\]
The last two inequalities follow from \eqref{eq:remainder-row} and the estimate~\eqref{eq:row} in Proposition~\ref{prop:row}, respectively.
The rank of $\Pi_q^{\mathrm{in}}$ is the supremum of the sizes $N$ of these families.
This gives \eqref{eq:rank} for $\Pi_q^{\mathrm{in}}$ after absorbing the factor $32$ into $C_{n,R}$.
Applying the same argument to $\mathcal{T}_q^*=\mathcal{T}_{\mathrm{ref}}^*+\mathcal{R}_q^*$, with $\|\mathcal{R}_q^*\|_{\mathfrak S_2}=\|\mathcal{R}_q\|_{\mathfrak S_2}$, proves the bound for $\Pi_q^{\mathrm{out}}$.

\medskip
\noindent\textit{Step 4.
Construction of the partial inverse.}
We now show that $\mathcal{T}_q$ is invertible from $\ker\Pi_q^{\mathrm{in}}$ onto $\ker\Pi_q^{\mathrm{out}}$ and use this restriction to define $\mathcal{W}_q$.

The bounded Borel functional calculus~\cite[Sections~VII.2--VII.3]{reed1980functional}, applied to the indicator function $\mathbf{1}_{[0,1/16)}$, gives
\[
\mathcal{T}_q\mathbf{1}_{[0,1/16)}(\mathcal{T}_q^*\mathcal{T}_q)
=\mathbf{1}_{[0,1/16)}(\mathcal{T}_q\mathcal{T}_q^*)\mathcal{T}_q.
\]
By \eqref{eq:bad-projections}, this is $\mathcal{T}_q\Pi_q^{\mathrm{in}}=\Pi_q^{\mathrm{out}}\mathcal{T}_q$.
Thus $\mathcal{T}_q$ maps $\ker\Pi_q^{\mathrm{in}}$ into $\ker\Pi_q^{\mathrm{out}}$.

The cutoff gives $\|\mathcal{T}_qz\|\geq\|z\|/4$ on $\ker\Pi_q^{\mathrm{in}}$, proving injectivity of this restriction.
The same cutoff makes $\mathcal{T}_q\mathcal{T}_q^*$ invertible on $\ker\Pi_q^{\mathrm{out}}$, so every $y\in\ker\Pi_q^{\mathrm{out}}$ can be written as $y=\mathcal{T}_q\mathcal{T}_q^*v$ with $v\in\ker\Pi_q^{\mathrm{out}}$.
The adjoint projection identity gives $\mathcal{T}_q^*v\in\ker\Pi_q^{\mathrm{in}}$, proving surjectivity.

Define $\mathcal{W}_q$ as the inverse of this restriction on $\ker\Pi_q^{\mathrm{out}}$, extended by zero on $\operatorname{ran}\Pi_q^{\mathrm{out}}$.
The lower bound gives $\|\mathcal{W}_q\|\leq4$, and the definition gives $\mathcal{T}_q\mathcal{W}_q=I-\Pi_q^{\mathrm{out}}$.
Since $\mathcal{T}_q$ maps $\operatorname{ran}\Pi_q^{\mathrm{in}}$ into $\operatorname{ran}\Pi_q^{\mathrm{out}}$, the product $\mathcal{W}_q\mathcal{T}_q$ vanishes on $\operatorname{ran}\Pi_q^{\mathrm{in}}$ and is the identity on $\ker\Pi_q^{\mathrm{in}}$.
This proves \eqref{eq:partial-identities}, and the inverse gives the unique solution $z=\mathcal{W}_qy$ in $\ker\Pi_q^{\mathrm{in}}$ for every $y\in\ker\Pi_q^{\mathrm{out}}$.
\end{proof}

\subsection{Exact solutions from compatible weights}\label{sec:sources}\label{sec:response}
The weight $g$ in Definition~\ref{def:incident-response} gives the harmonic function $\exp(\tau a\cdot x)\mathcal{E}_ag$ in \eqref{eq:superposition}.
We seek a correction that turns it into a solution of the Schr\"odinger equation~\eqref{eq:conductivity-potentials} with potential $q=q_j$.
Writing $w=\exp(\tau a\cdot x)(\mathcal{E}_ag+\mathcal{A}z)$, equations~\eqref{eq:conjugated-correction} and \eqref{eq:exact-from-T} reduce this task to $\mathcal{T}_qz=-\mathcal{A}(q\mathcal{E}_ag)$.
Remark~\ref{rem:partial-inverse-meaning} shows that applying $\mathcal{W}_q$ leaves the $L^2$ residual $\Pi_q^{\mathrm{out}}\mathcal{A}(q\mathcal{E}_ag)$.
We choose weights for which this residual vanishes.

For fixed $a,\tau$, define the linear map $\mathcal{C}_q$ so that the $L^2$ correction equation reads $\mathcal{T}_qz=-\mathcal{C}_qg$:
\begin{equation}\label{eq:source}
\mathcal{C}_qg:=\mathcal{A}(q\mathcal{E}_ag),\qquad
\mathcal{C}_q:L^2(\mathbb{S}(a^\perp),\mathrm{d}\sigma_a)\longrightarrow L^2(\mathbb{R}^n).
\end{equation}
The form bound~\eqref{eq:form-bound} and Lemma~\ref{lem:A-maps} show that $b\mapsto\mathcal{A}(q\exp(\mathrm{i}\tau b\cdot x))$ is continuous into $L^2(\mathbb{R}^n)$.
By \eqref{eq:incident}, $\mathcal{C}_qg$ is obtained by integrating this $L^2$-valued kernel against $g(b)$; the continuity above justifies passing $q$ and $\mathcal{A}$ through the integral.
Because the sphere is compact, the kernel is square-integrable.
The Hilbert--Schmidt kernel criterion~\cite[Theorem~2.11]{simon2005trace} applied to this square-integrable kernel gives $\mathcal{C}_q\in\mathfrak S_2$ and the norm identity
\begin{equation}\label{eq:source-kernel}
\|\mathcal{C}_q\|_{\mathfrak S_2}^2
=\int_{\mathbb{S}(a^\perp)}\|\mathcal{A}(q\exp(\mathrm{i}\tau b\cdot x))\|_2^2\,\mathrm{d}\sigma_a(b).
\end{equation}

For $j=1,2$, write $\mathcal{T}_j=\mathcal{T}_{q_j}$ and $\mathcal{C}_j=\mathcal{C}_{q_j}$, with $q_j$ from Definition~\ref{def:schrodinger-potentials}.
Take $\tau$ above both thresholds in Proposition~\ref{prop:partial}.
For $j=1,2$, write $\mathcal{W}_j=\mathcal{W}_{q_j}$, $\Pi_j^{\mathrm{in}}=\Pi_{q_j}^{\mathrm{in}}$, and $\Pi_j^{\mathrm{out}}=\Pi_{q_j}^{\mathrm{out}}$ for the corresponding partial inverses and projections.
By \eqref{eq:symbol-intro}, $d_{-a}=\overline{d_a}$, so $|d_{-a}|=|d_a|$ and $\mathcal{A}$ is the same for both directions.
Since $q_2$ is real, the adjoint $\mathcal{T}_2^*$ therefore represents the $L^2$ correction equation for $d_{-a}(D)+q_2$.
Equation~\eqref{eq:partial-identities} and its adjoint give $\mathcal{T}_1\mathcal{W}_1=I-\Pi_1^{\mathrm{out}}$ and $\mathcal{T}_2^*\mathcal{W}_2^*=I-\Pi_2^{\mathrm{in}}$.
Thus the first correction leaves the source component $\Pi_1^{\mathrm{out}}\mathcal{C}_1g$, while the second leaves $\Pi_2^{\mathrm{in}}\mathcal{C}_2h$.
We therefore restrict the weights so that the corresponding source component vanishes.
\begin{definition}\label{def:compatible-weights}
For fixed $a,\tau$, define subspaces of $L^2(\mathbb{S}(a^\perp),\mathrm{d}\sigma_a)$ by
\begin{equation}\label{eq:compatible}
\mathcal{G}_1=\ker(\Pi_1^{\mathrm{out}}\mathcal{C}_1),\qquad
\mathcal{G}_2=\ker(\Pi_2^{\mathrm{in}}\mathcal{C}_2),\qquad
\mathcal{G}=\mathcal{G}_1\cap\mathcal{G}_2.
\end{equation}
Members of $\mathcal{G}_1$ are \emph{compatible with the first correction}, and members of $\mathcal{G}_2$ are \emph{compatible with the second (adjoint) correction}.
Members of $\mathcal{G}$ are compatible with both.
\end{definition}

By the correction equation~\eqref{eq:conjugated-correction} and the source definition~\eqref{eq:source}, the first and adjoint correction equations have sources $-\mathcal{C}_1g$ and $-\mathcal{C}_2h$.
Proposition~\ref{prop:partial} provides the partial inverse $\mathcal{W}_1$ for the first equation; taking the adjoint of \eqref{eq:partial-identities} provides $\mathcal{W}_2^*$ for the second.
Using $r=\mathcal{A}z$ from Lemma~\ref{lem:bounded-correction}, these partial inverses give the candidate corrections $-\mathcal{A}\mathcal{W}_1\mathcal{C}_1g$ and $-\mathcal{A}\mathcal{W}_2^*\mathcal{C}_2h$.
Adding them to the free oscillatory factors from \eqref{eq:incident} defines the maps
\begin{equation}\label{eq:corrected}
\mathcal{X}_1=\mathcal{E}_a-\mathcal{A}\mathcal{W}_1\mathcal{C}_1,\qquad
\mathcal{X}_2=\mathcal{E}_a-\mathcal{A}\mathcal{W}_2^*\mathcal{C}_2.
\end{equation}

Let $\mathcal{P}$ be the orthogonal projection of the weight space onto $\mathcal{G}$; then $\operatorname{rank}(I-\mathcal{P})$ is the codimension of the common compatible subspace.
The next proposition proves that the corrected factors in \eqref{eq:corrected} give exact solutions for the compatible weights of Definition~\ref{def:compatible-weights}, and bounds $\operatorname{rank}(I-\mathcal{P})$.

\begin{proposition}\label{prop:compatible-solutions}
Let $n\geq4$, fix $R>2$, and let $h_1,h_2$ satisfy \eqref{eq:compact-model}, with potentials $q_1,q_2$ from Definition~\ref{def:schrodinger-potentials}.
For every $a\in\mathbb{S}^{n-1}$ and $\tau\geq\max\{\tau_0(q_1),\tau_0(q_2)\}$, the maps $\mathcal{X}_j$ in \eqref{eq:corrected} are bounded from $L^2(\mathbb{S}(a^\perp),\mathrm{d}\sigma_a)$ to $H^1(B)$.
If $g\in\mathcal{G}_1$ and $h\in\mathcal{G}_2$ as in Definition~\ref{def:compatible-weights}, then
\begin{equation}\label{eq:exact-solutions}
w_1=\exp(\tau a\cdot x)\mathcal{X}_1g,
\qquad w_2=\exp(-\tau a\cdot x)\overline{\mathcal{X}_2h}
\end{equation}
belong to $H^1(B)$ and are weak solutions of $(-\Delta+q_j)w_j=0$ in $B$ for $j=1,2$.
Consequently, $\langle q_1-q_2,w_1w_2\rangle=0$ by \eqref{eq:alessandrini}.
Moreover, with $\Gamma_{12}(\tau)=\Gamma_{q_1}(\tau)+\Gamma_{q_2}(\tau)$,
\begin{equation}\label{eq:codimension}
\operatorname{rank}(I-\mathcal{P})
\leq\operatorname{rank}\Pi_1^{\mathrm{out}}+\operatorname{rank}\Pi_2^{\mathrm{in}}
\leq C_{n,R}\tau^{n-2}\Gamma_{12}(\tau).
\end{equation}
The constant $C_{n,R}$ depends only on $n$ and $R$.
In particular, $\operatorname{rank}(I-\mathcal{P})=o(\tau^{n-2})$ as $\tau\to\infty$, uniformly in $a$.
\end{proposition}

\begin{proof}
Cauchy--Schwarz in \eqref{eq:incident} shows that $\mathcal{E}_a$ maps the weight space boundedly into $H^1(B)$.
The source operators $\mathcal{C}_j$ are Hilbert--Schmidt by \eqref{eq:source-kernel}, and Lemma~\ref{lem:A-maps} maps $L^2$ into $H^1(B)$ through $\mathcal{A}$.
Together with $\|\mathcal{W}_j\|\leq4$ from \eqref{eq:partial-identities}, these facts prove the mapping assertion for $\mathcal{X}_j$.

For arbitrary weights $g,h$, \eqref{eq:partial-identities} and its adjoint give
\begin{equation}\label{eq:source-residuals}
\mathcal{T}_1(-\mathcal{W}_1\mathcal{C}_1g)=-\mathcal{C}_1g+\Pi_1^{\mathrm{out}}\mathcal{C}_1g,\quad
\mathcal{T}_2^*(-\mathcal{W}_2^*\mathcal{C}_2h)=-\mathcal{C}_2h+\Pi_2^{\mathrm{in}}\mathcal{C}_2h.
\end{equation}
For $g\in\mathcal{G}_1$ and $h\in\mathcal{G}_2$, \eqref{eq:compatible} makes the projection terms in \eqref{eq:source-residuals} vanish.
Using \eqref{eq:source}, we obtain the $L^2$ equations
\begin{equation}\label{eq:compatible-L2}
\mathcal{T}_1(-\mathcal{W}_1\mathcal{C}_1g)=-\mathcal{A}(q_1\mathcal{E}_ag),\quad
\mathcal{T}_2^*(-\mathcal{W}_2^*\mathcal{C}_2h)=-\mathcal{A}(q_2\mathcal{E}_ah).
\end{equation}
The form bound~\eqref{eq:form-bound} gives $q_1\mathcal{E}_ag,q_2\mathcal{E}_ah\in H^{-1}_E$.
Lemma~\ref{lem:bounded-correction} gives the identities that turn these $L^2$ equations into correction equations:
\[
(d_a(D)+q_1)\mathcal{A}z=|d_a(D)|^{1/2}\mathcal{T}_1z,\qquad |d_a(D)|^{1/2}\mathcal{A}s=s.
\]
The second identity removes $\mathcal{A}$ from the source terms.
Because $q_2$ is real and $|d_{-a}|=|d_a|$, the first identity also holds with $d_{-a}(D)+q_2$ and $\mathcal{T}_2^*$ in place of $d_a(D)+q_1$ and $\mathcal{T}_1$.
Applying these identities to \eqref{eq:compatible-L2} gives
\begin{equation}\label{eq:compatible-correction}
(d_a(D)+q_1)(-\mathcal{A}\mathcal{W}_1\mathcal{C}_1g)=-q_1\mathcal{E}_ag,\quad
(d_{-a}(D)+q_2)(-\mathcal{A}\mathcal{W}_2^*\mathcal{C}_2h)=-q_2\mathcal{E}_ah.
\end{equation}
By \eqref{eq:incident} and \eqref{eq:symbol-intro}, $d_a(D)\mathcal{E}_ag=d_{-a}(D)\mathcal{E}_ah=0$.
Adding these free factors to the corrections in \eqref{eq:compatible-correction}, as prescribed by \eqref{eq:corrected}, cancels the source terms and gives
\begin{equation}\label{eq:compatible-conjugated}
(d_a(D)+q_1)\mathcal{X}_1g=0,\qquad (d_{-a}(D)+q_2)\mathcal{X}_2h=0.
\end{equation}
Undoing the conjugation in \eqref{eq:compatible-conjugated} using \eqref{eq:conjugated-equation}, with $a$ replaced by $-a$ for the second factor, gives $\exp(\tau a\cdot x)\mathcal{X}_1g$ and $\exp(-\tau a\cdot x)\mathcal{X}_2h$ as Schr\"odinger solutions.
Both belong to $H^1(B)$ because $B$ is bounded and $\mathcal{X}_1g,\mathcal{X}_2h\in H^1(B)$.
Since $q_2$ is real, complex conjugation preserves the second equation and gives $w_2$ in \eqref{eq:exact-solutions}.
The resulting pair of solutions satisfies \eqref{eq:alessandrini}.

For arbitrary weights, the same calculation starting from \eqref{eq:source-residuals} retains the projection terms and gives
\begin{equation}\label{eq:residuals}
(d_a(D)+q_1)\mathcal{X}_1=|d_a(D)|^{1/2}\Pi_1^{\mathrm{out}}\mathcal{C}_1,\quad
(d_{-a}(D)+q_2)\mathcal{X}_2=|d_a(D)|^{1/2}\Pi_2^{\mathrm{in}}\mathcal{C}_2.
\end{equation}
These identities hold as operators from the weight space to $H^{-1}_{\mathrm{loc}}(\mathbb{R}^n)$.

Finally, the map $g\mapsto(\Pi_1^{\mathrm{out}}\mathcal{C}_1g,\Pi_2^{\mathrm{in}}\mathcal{C}_2g)$ has kernel $\mathcal{G}$ and range contained in $\operatorname{ran}\Pi_1^{\mathrm{out}}\oplus\operatorname{ran}\Pi_2^{\mathrm{in}}$.
Hence $\operatorname{rank}(I-\mathcal{P})=\operatorname{codim}\mathcal{G}\leq\operatorname{rank}\Pi_1^{\mathrm{out}}+\operatorname{rank}\Pi_2^{\mathrm{in}}$.
Applying \eqref{eq:rank} to $q_1$ and $q_2$ proves \eqref{eq:codimension}, and Lemma~\ref{lem:gamma-decay} gives the stated $o(\tau^{n-2})$ bound.
\end{proof}

\section{Averaged trace estimates}\label{sec:trace-method}
In this section we prove Proposition~\ref{prop:trace-decay}: for $n\geq4$, equality of the Dirichlet-to-Neumann maps implies $\mathbb E_a\|\mathcal Q(a,\tau)\|_{\mathfrak S_1}\to0$ as $\tau\to\infty$.
Let $\mathcal P$ be the orthogonal projection onto the common compatible weight space $\mathcal G$ from Proposition~\ref{prop:compatible-solutions}, and decompose $\mathcal Q=\mathcal Q(a,\tau)$ into the compression $\mathcal P\mathcal Q\mathcal P$ and the remainder $\mathcal Q-\mathcal P\mathcal Q\mathcal P$.
In Subsection~\ref{sec:response-identity}, we use Alessandrini's identity~\eqref{eq:alessandrini} for the exact CGO solutions on $\mathcal G$ and the uniform bounds for the partial inverses to control $\|\mathcal P\mathcal Q\mathcal P\|_{\mathfrak S_1}$ by $\sum_{j=1}^2\|\mathcal C_j\|_{\mathfrak S_2}^2$.
Subsection~\ref{sec:averages} shows that the spherical average of this sum tends to zero and that the spherical root mean square $(\mathbb E_a\|\mathcal Q(a,\tau)\|_{\mathfrak S_2}^2)^{1/2}$ is $O(\tau^{-(n-2)/2})$.
In Subsection~\ref{sec:trace}, we use the finite-rank trace norm inequality to bound the remainder $\mathcal Q-\mathcal P\mathcal Q\mathcal P$. Its rank is at most $2\operatorname{rank}(I-\mathcal P)=o(\tau^{n-2})$ uniformly in $a$, so the square root of this rank bound times the root mean square Hilbert--Schmidt bound is $o(1)$.

\subsection{The exact response identity}\label{sec:response-identity}
Fix $a$ and $\tau$ as in Proposition~\ref{prop:compatible-solutions}.
To relate the compatible compression $\mathcal{P}\mathcal{Q}\mathcal{P}$ to the CGO corrections, we first compare the full response pairing with the free response on the entire weight space $L^2(\mathbb{S}(a^\perp),\mathrm{d}\sigma_a)$.
For arbitrary weights $g,h$, define the bounded operator $\mathcal{D}$ on $L^2(\mathbb{S}(a^\perp),\mathrm{d}\sigma_a)$ by
\[
(\mathcal{D}g,h)_{L^2(\mathbb{S}(a^\perp),\mathrm{d}\sigma_a)}=\langle q_1-q_2,(\mathcal{X}_1g)\overline{\mathcal{X}_2h}\rangle,
\]
where boundedness follows from the $L^2\to H^1(B)$ bound for $\mathcal{X}_j$ in Proposition~\ref{prop:compatible-solutions} and the form bound~\eqref{eq:form-bound}.
When the Dirichlet-to-Neumann maps agree, \eqref{eq:alessandrini} implies that $(\mathcal{D}g,h)=0$ whenever $g\in\mathcal{G}_1$ and $h\in\mathcal{G}_2$; in particular, $\mathcal{P}\mathcal{D}\mathcal{P}=0$.

Taking the uncorrected factor from each map in \eqref{eq:corrected} gives $(\mathcal{E}_ag)\overline{\mathcal{E}_ah}$, whose pairing with $q_1-q_2$ is the free response $(\mathcal{Q}g,h)$ by \eqref{eq:Q-form}.
Using the source operators $\mathcal{C}_j=\mathcal{C}_{q_j}$ from \eqref{eq:source} and the corresponding partial inverses $\mathcal{W}_j$, we express $\mathcal{D}-\mathcal{Q}$ as the sum of one correction operator from each potential and a finite-rank remainder that vanishes upon compression by $\mathcal{P}$.

\begin{lemma}\label{lem:nonlinear}
For $a$ and $\tau$ as in Proposition~\ref{prop:compatible-solutions}, define the correction response operators by
\[
\mathcal{N}_j:=\mathcal{C}_j^*\mathcal{W}_j\mathcal{C}_j,
\qquad j=1,2.
\]
Then, as operators on the full weight space $L^2(\mathbb{S}(a^\perp),\mathrm{d}\sigma_a)$,
\begin{equation}\label{eq:full-response}
\mathcal{D}=\mathcal{Q}-\mathcal{N}_1+\mathcal{N}_2
+\mathcal{C}_2^*(\Pi_2^{\mathrm{in}}\mathcal{W}_1-\mathcal{W}_2\Pi_1^{\mathrm{out}})\mathcal{C}_1.
\end{equation}
The last term has finite rank.
If the Dirichlet-to-Neumann maps agree, let $\mathcal{P}$ be the orthogonal projection onto the common compatible subspace $\mathcal{G}=\mathcal{G}_1\cap\mathcal{G}_2$ of Definition~\ref{def:compatible-weights}.
Then
\begin{equation}\label{eq:compressed-response}
\mathcal{P}\mathcal{Q}\mathcal{P}=\mathcal{P}(\mathcal{N}_1-\mathcal{N}_2)\mathcal{P}.
\end{equation}
The compression satisfies the trace norm bound
\begin{equation}\label{eq:compatible-trace-bound}
\|\mathcal{P}\mathcal{Q}\mathcal{P}\|_{\mathfrak S_1}
\leq4\sum_{j=1}^2\|\mathcal{C}_j\|_{\mathfrak S_2}^2.
\end{equation}
\end{lemma}

\begin{proof}
For $g,h\in L^2(\mathbb{S}(a^\perp),\mathrm{d}\sigma_a)$, expanding the product of the corrected factors from \eqref{eq:corrected} gives
\begin{equation}\label{eq:response-pairing-expansion}
\begin{aligned}
(\mathcal{D}g,h)={}&\langle q_1-q_2,\mathcal{E}_ag\,\overline{\mathcal{E}_ah}\rangle
-\langle q_1-q_2,\mathcal{A}\mathcal{W}_1\mathcal{C}_1g\,\overline{\mathcal{E}_ah}\rangle\\
&-\langle q_1-q_2,\mathcal{E}_ag\,\overline{\mathcal{A}\mathcal{W}_2^*\mathcal{C}_2h}\rangle
+\langle q_1-q_2,\mathcal{A}\mathcal{W}_1\mathcal{C}_1g\,\overline{\mathcal{A}\mathcal{W}_2^*\mathcal{C}_2h}\rangle.
\end{aligned}
\end{equation}
The factors belong locally to $H^1$, so all pairings are defined by \eqref{eq:q-form}.
The first term is $(\mathcal{Q}g,h)$ by \eqref{eq:Q-form}.
By the form bound~\eqref{eq:form-bound} and the support condition on the potentials, $(q_1-q_2)w\in H^{-1}_E$ for each local $H^1$ factor $w$ used in the correction terms.
For all three correction terms, use the duality formula $\langle s,\overline{\mathcal{A}z}\rangle=(\mathcal{A}s,z)_{L^2}$ from \eqref{eq:A-duality}.
The definitions \eqref{eq:source} and \eqref{eq:T-definition} give
\[
\mathcal{A}(q_1-q_2)\mathcal{E}_a=\mathcal{C}_1-\mathcal{C}_2,
\qquad
\mathcal{A}(q_1-q_2)\mathcal{A}=\mathcal{T}_1-\mathcal{T}_2.
\]
For the first correction term, take $s=(q_1-q_2)\mathcal{E}_ah$ and $z=\mathcal{W}_1\mathcal{C}_1g$, and conjugate the resulting identity using the reality of $q_1-q_2$.
For the second and third correction terms, take $z=\mathcal{W}_2^*\mathcal{C}_2h$ with $s=(q_1-q_2)\mathcal{E}_ag$ and $s=(q_1-q_2)\mathcal{A}\mathcal{W}_1\mathcal{C}_1g$, respectively.
The three correction terms in \eqref{eq:response-pairing-expansion}, in order, are therefore
\[
\begin{gathered}
((\mathcal{C}_1-\mathcal{C}_2)^*\mathcal{W}_1\mathcal{C}_1g,h),\\
(\mathcal{C}_2^*\mathcal{W}_2(\mathcal{C}_1-\mathcal{C}_2)g,h),\\
(\mathcal{C}_2^*\mathcal{W}_2(\mathcal{T}_1-\mathcal{T}_2)\mathcal{W}_1\mathcal{C}_1g,h).
\end{gathered}
\]
Substituting these pairings into the expansion \eqref{eq:response-pairing-expansion} gives
\begin{equation}\label{eq:four-terms}
\mathcal{D}=\mathcal{Q}
-(\mathcal{C}_1-\mathcal{C}_2)^*\mathcal{W}_1\mathcal{C}_1
-\mathcal{C}_2^*\mathcal{W}_2(\mathcal{C}_1-\mathcal{C}_2)
+\mathcal{C}_2^*\mathcal{W}_2(\mathcal{T}_1-\mathcal{T}_2)\mathcal{W}_1\mathcal{C}_1.
\end{equation}
Expanding the source differences in \eqref{eq:four-terms} yields the contributions $-\mathcal{C}_1^*\mathcal{W}_1\mathcal{C}_1=-\mathcal{N}_1$ and $\mathcal{C}_2^*\mathcal{W}_2\mathcal{C}_2=\mathcal{N}_2$, which depend on $q_1$ and $q_2$, respectively.
The remaining correction terms share the left factor $\mathcal{C}_2^*$ and the right factor $\mathcal{C}_1$.
The partial-inverse identities \eqref{eq:partial-identities} of Proposition~\ref{prop:partial} give $\mathcal{T}_1\mathcal{W}_1=I-\Pi_1^{\mathrm{out}}$ and $\mathcal{W}_2\mathcal{T}_2=I-\Pi_2^{\mathrm{in}}$.
Substituting them into the combined middle factor gives
\[
\begin{split}
\mathcal{W}_1-\mathcal{W}_2+\mathcal{W}_2(\mathcal{T}_1-\mathcal{T}_2)\mathcal{W}_1
&=\mathcal{W}_1-\mathcal{W}_2+\mathcal{W}_2(I-\Pi_1^{\mathrm{out}})-(I-\Pi_2^{\mathrm{in}})\mathcal{W}_1\\
&=\Pi_2^{\mathrm{in}}\mathcal{W}_1-\mathcal{W}_2\Pi_1^{\mathrm{out}}.
\end{split}
\]
This proves \eqref{eq:full-response}, whose last term has finite rank because $\Pi_2^{\mathrm{in}}$ and $\Pi_1^{\mathrm{out}}$ are finite-rank projections.

By \eqref{eq:source-kernel}, $\mathcal{C}_j$ and its adjoint are Hilbert--Schmidt with the same norm.
Proposition~\ref{prop:partial} gives $\|\mathcal{W}_j\|\leq4$ in \eqref{eq:partial-identities}, so $\mathcal{W}_j\mathcal{C}_j$ is also Hilbert--Schmidt.
The product $\mathcal{N}_j$ is therefore trace class, and the Hilbert--Schmidt product inequality gives
\begin{equation}\label{eq:N-bound}
\|\mathcal{N}_j\|_{\mathfrak S_1}
\leq\|\mathcal{C}_j^*\|_{\mathfrak S_2}\|\mathcal{W}_j\mathcal{C}_j\|_{\mathfrak S_2}
\leq\|\mathcal{W}_j\|\|\mathcal{C}_j\|_{\mathfrak S_2}^2
\leq4\|\mathcal{C}_j\|_{\mathfrak S_2}^2.
\end{equation}
The free response $\mathcal{Q}$ is trace class by the smooth kernel criterion recalled in Subsection~\ref{sec:incident-response}.
Thus \eqref{eq:full-response} also gives $\mathcal{D}\in\mathfrak S_1$.

Suppose that the Dirichlet-to-Neumann maps agree and take $g\in\mathcal{G}_1$ and $h\in\mathcal{G}_2$.
The real exponentials in the exact solutions \eqref{eq:exact-solutions} of Proposition~\ref{prop:compatible-solutions} cancel in $w_1w_2$.
Alessandrini's identity~\eqref{eq:alessandrini} then gives $(\mathcal{D}g,h)=\langle q_1-q_2,w_1w_2\rangle=0$.
The compatibility conditions \eqref{eq:compatible} give $\Pi_1^{\mathrm{out}}\mathcal{C}_1g=0$ and $\Pi_2^{\mathrm{in}}\mathcal{C}_2h=0$.
Using also the self-adjointness of $\Pi_2^{\mathrm{in}}$, we obtain
\[
\begin{aligned}
(\mathcal{C}_2^*\Pi_2^{\mathrm{in}}\mathcal{W}_1\mathcal{C}_1g,h)
&=(\mathcal{W}_1\mathcal{C}_1g,\Pi_2^{\mathrm{in}}\mathcal{C}_2h)=0,\\
(\mathcal{C}_2^*\mathcal{W}_2\Pi_1^{\mathrm{out}}\mathcal{C}_1g,h)
&=(\mathcal{W}_2\Pi_1^{\mathrm{out}}\mathcal{C}_1g,\mathcal{C}_2h)=0.
\end{aligned}
\]
Substituting the three vanishing pairings into \eqref{eq:full-response} gives
\begin{equation}\label{eq:compatible-pairing}
((\mathcal{Q}-\mathcal{N}_1+\mathcal{N}_2)g,h)_{L^2(\mathbb{S}(a^\perp),\mathrm{d}\sigma_a)}=0,
\qquad g\in\mathcal{G}_1,\quad h\in\mathcal{G}_2.
\end{equation}
For arbitrary weights $u,v$, take $g=\mathcal{P}u$ and $h=\mathcal{P}v$ in \eqref{eq:compatible-pairing}.
Since $\mathcal{P}^*=\mathcal{P}$, this proves \eqref{eq:compressed-response}.
The triangle inequality and the trace ideal bound with $\|\mathcal{P}\|\leq1$ give
\[
\|\mathcal{P}\mathcal{Q}\mathcal{P}\|_{\mathfrak S_1}
\leq\sum_{j=1}^2\|\mathcal{P}\mathcal{N}_j\mathcal{P}\|_{\mathfrak S_1}
\leq\sum_{j=1}^2\|\mathcal{N}_j\|_{\mathfrak S_1}.
\]
Together with \eqref{eq:N-bound}, this proves \eqref{eq:compatible-trace-bound}.
\end{proof}

\subsection{Spherical averages of source and response operators}\label{sec:averages}
For each $a$ and $\tau$ as in Proposition~\ref{prop:compatible-solutions}, write the free response as
\begin{equation}\label{eq:response-decomposition}
\mathcal{Q}=\mathcal{P}\mathcal{Q}\mathcal{P}+(\mathcal{Q}-\mathcal{P}\mathcal{Q}\mathcal{P}).
\end{equation}
Lemma~\ref{lem:source-average} estimates the source averages $\mathbb E_a\|\mathcal{C}_j\|_{\mathfrak S_2}^2$, $j=1,2$, which control the contribution of $\mathcal{P}\mathcal{Q}\mathcal{P}$ through the compression bound~\eqref{eq:compatible-trace-bound} in Lemma~\ref{lem:nonlinear}.
Lemma~\ref{lem:response-average} bounds the free response average $\mathbb E_a\|\mathcal{Q}\|_{\mathfrak S_2}^2$ by $O(\tau^{-(n-2)})$, providing the decay needed to control the remainder $\mathcal{Q}-\mathcal{P}\mathcal{Q}\mathcal{P}$.
We average over $a$ at fixed $\tau$ with respect to $\mu$ on $\mathbb{S}^{n-1}$, as in Definition~\ref{def:direction-average}.

\begin{lemma}\label{lem:source-average}
Let $n\geq4$ and let $q=\operatorname{div}F+V$ be real, with $F\in L^n_E(\mathbb{R}^n;\mathbb{R}^n)$ and $V\in L^{n/2}_E(\mathbb{R}^n;\mathbb{R})$ for a compact set $E$.
For every $\tau\geq2$, the source operator $\mathcal{C}_q$ from \eqref{eq:source} satisfies
\begin{equation}\label{eq:source-average}
\tfrac12\Gamma_q(\tau)
\leq\mathbb E_a\|\mathcal{C}_q(a,\tau)\|_{\mathfrak S_2}^2
\leq C_n\Gamma_q(\tau),
\end{equation}
where $C_n>0$ depends only on $n$ and $\Gamma_q(\tau)$ is defined in \eqref{eq:gamma-intro}.
\end{lemma}

\begin{proof}

\medskip
\noindent\textit{Step 1.
Fourier representation of the source norm.}
Since $\mathcal{A}=|d_a(D)|^{-1/2}$, Plancherel and the Fourier shift give
\[
\|\mathcal{A}(q\exp(\mathrm{i}\tau b\cdot x))\|_2^2
=\int\frac{|\widehat q(\xi-\tau b)|^2}{|d_a(\xi)|}\,\mathrm{d}\xi.
\]
Set $\eta=\xi-\tau b$.
The symbol \eqref{eq:symbol-intro}, together with $|b|=1$ and $a\cdot b=0$, gives
\[
d_a(\eta+\tau b)=|\eta|^2+2\tau b\cdot\eta-2\mathrm{i}\tau a\cdot\eta.
\]
Substituting into the source norm formula \eqref{eq:source-kernel} and exchanging the $b$ and $\eta$ integrals by Tonelli's theorem~\cite[Theorem~2.37(a)]{folland1999real} gives
\begin{equation}\label{eq:source-integral}
\|\mathcal{C}_q\|_{\mathfrak S_2}^2
=\int|\widehat q(\eta)|^2
\int_{\mathbb{S}(a^\perp)}
\frac{\mathrm{d}\sigma_a(b)}{\big||\eta|^2+2\tau b\cdot\eta-2\mathrm{i}\tau a\cdot\eta\big|}
\,\mathrm{d}\eta.
\end{equation}

We average over $a\in\mathbb{S}^{n-1}$ with the normalized surface measure $\mu$ from Definition~\ref{def:direction-average}.
For each fixed $a$, we integrate over $b\in\mathbb{S}(a^\perp)$ with the normalized surface measure $\sigma_a$ from Definition~\ref{def:incident-response}.
Applying Tonelli again to the nonnegative integrand in \eqref{eq:source-integral}, we average the reciprocal denominator over $a,b$ before integrating in $\eta$.
To obtain the upper bound in \eqref{eq:source-average}, it suffices to bound this average by $C_n/(|\eta|^2+\tau|\eta|)$, matching the weight in the definition \eqref{eq:gamma-intro} of $\Gamma_q(\tau)$.

\medskip
\noindent\textit{Step 2.
A uniform bound for the spherical average.}
Fix $\eta\in\mathbb{R}^n\setminus\{0\}$.
The normalized projections of $\eta$ onto $b$ and $a$ form the two-dimensional vector
\[
y=(y_1,y_2)=\left(\frac{b\cdot\eta}{|\eta|},\frac{a\cdot\eta}{|\eta|}\right)\in\mathbb{R}^2.
\]
Since $a,b$ are orthonormal, $y_1^2+y_2^2\leq1$, so $y$ lies in the closed unit disk $\{(y_1,y_2)\in\mathbb{R}^2:y_1^2+y_2^2\leq1\}$.
Substituting $b\cdot\eta=|\eta|y_1$ and $a\cdot\eta=|\eta|y_2$ rewrites the denominator as
\[
|d_a(\eta+\tau b)|
=2\tau|\eta|\sqrt{(y_1+|\eta|/(2\tau))^2+y_2^2}.
\]
The pair $(b,a)$ is uniformly distributed among orthonormal pairs, so rotation invariance gives $y$ the same distribution as the first two coordinates of a uniform point on $\mathbb{S}^{n-1}$.
The parametrization $(y,\sqrt{1-|y|^2}\,\omega)$, with $|y|<1$ and $\omega\in\mathbb{S}^{n-3}$, gives the surface element $(1-|y|^2)^{(n-4)/2}\,\mathrm{d}y\,\mathrm{d}S(\omega)$.
The integral of this weight over the unit disk is $2\pi/(n-2)$, so $y$ has density
\begin{equation}\label{eq:disk-density}
\frac{n-2}{2\pi}(1-y_1^2-y_2^2)^{(n-4)/2}
\mathbf{1}_{\{y_1^2+y_2^2<1\}}
\end{equation}
with respect to planar Lebesgue measure $\mathrm{d}y=\mathrm{d}y_1\,\mathrm{d}y_2$.
Averaging the reciprocal denominator with the density \eqref{eq:disk-density} gives
\begin{equation}\label{eq:source-denominator-average}
\mathbb E_a\int_{\mathbb{S}(a^\perp)}
\frac{\mathrm{d}\sigma_a(b)}{|d_a(\eta+\tau b)|}
=\frac{n-2}{4\pi\tau|\eta|}
\int_{|y|<1}
\frac{(1-|y|^2)^{(n-4)/2}}
{\sqrt{(y_1+|\eta|/(2\tau))^2+y_2^2}}\,\mathrm{d}y.
\end{equation}
For $n\geq4$, the weight $(1-|y|^2)^{(n-4)/2}$ is at most $1$.
Dropping this weight leaves the integral of the reciprocal distance from $y$ to $(-|\eta|/(2\tau),0)$; we bound this integral uniformly in the location of that point.

For $y_0\in\mathbb{R}^2$ and $t>0$, the condition $|y-y_0|^{-1}>t$ defines the disk $B_{1/t}(y_0)$ of radius $1/t$ centered at $y_0$.
The layer-cake formula~\cite[Proposition~6.24]{folland1999real} therefore gives
\[
\int_{|y|<1}|y-y_0|^{-1}\,\mathrm{d}y
=\int_0^\infty|B_1(0)\cap B_{1/t}(y_0)|\,\mathrm{d}t,
\]
where $B_1(0)$ is the unit disk and the bars around sets denote planar area.
The intersection has area at most that of either disk, hence at most $\min\{\pi,\pi t^{-2}\}$, so
\[
\int_{|y|<1}|y-y_0|^{-1}\,\mathrm{d}y
\leq\int_0^1\pi\,\mathrm{d}t+\int_1^\infty\pi t^{-2}\,\mathrm{d}t=2\pi.
\]
Applying this bound to \eqref{eq:source-denominator-average} with $y_0=(-|\eta|/(2\tau),0)$ gives
\[
\mathbb E_a\int_{\mathbb{S}(a^\perp)}
\frac{\mathrm{d}\sigma_a(b)}{|d_a(\eta+\tau b)|}
\leq\frac{n-2}{4\pi\tau|\eta|}\,2\pi
=\frac{n-2}{2\tau|\eta|}.
\]

\medskip
\noindent\textit{Step 3.
Comparison with $\Gamma_q(\tau)$.}
If $|\eta|\leq4\tau$, then $|\eta|^2+\tau|\eta|\leq5\tau|\eta|$, so this upper bound is at most $5(n-2)/(2(|\eta|^2+\tau|\eta|))$.
If $|\eta|>4\tau$, the reverse triangle inequality and orthogonality give
\[
|d_a(\eta+\tau b)|\geq|\eta|^2-2\tau|\eta|\geq|\eta|^2/2.
\]
Thus its reciprocal is at most $2/|\eta|^2\leq5/(2(|\eta|^2+\tau|\eta|))$, uniformly in $a,b$, and the same bound holds after averaging.
The two ranges therefore give the upper bound $C_n/(|\eta|^2+\tau|\eta|)$ for the averaged reciprocal denominator, with $C_n=5(n-2)/2$.
Substituting this bound into the average of \eqref{eq:source-integral} gives
\[
\mathbb E_a\|\mathcal{C}_q\|_{\mathfrak S_2}^2
\leq C_n\int_{\mathbb{R}^n}
\frac{|\widehat q(\eta)|^2}{|\eta|^2+\tau|\eta|}\,\mathrm{d}\eta
=C_n\Gamma_q(\tau),
\]
where the last equality is the definition \eqref{eq:gamma-intro}.

For the lower bound, the triangle inequality and orthogonality give $|d_a(\eta+\tau b)|\leq|\eta|^2+2\tau|\eta|\leq2(|\eta|^2+\tau|\eta|)$ for every orthogonal pair $(a,b)$.
Its reciprocal is therefore at least $1/(2(|\eta|^2+\tau|\eta|))$.
Averaging in $a,b$ and integrating against $|\widehat q(\eta)|^2$ in the same way yields $\mathbb E_a\|\mathcal{C}_q\|_{\mathfrak S_2}^2\geq\tfrac12\Gamma_q(\tau)$, completing the proof.
\end{proof}

\begin{lemma}\label{lem:response-average}

Let $n\geq4$, and let $h_1,h_2$ satisfy the regularity and support conditions in \eqref{eq:compact-model}, with associated potentials $q_1,q_2$ from Definition~\ref{def:schrodinger-potentials}.
For every $\tau\geq2$, the free response operator $\mathcal{Q}(a,\tau)$ from Definition~\ref{def:incident-response} satisfies
\begin{equation}\label{eq:response-average}
\mathbb E_a\|\mathcal{Q}(a,\tau)\|_{\mathfrak S_2}^2
\leq c_n\tau^{-(n-2)}\|q_1-q_2\|_{\dot H^{-1}}^2,
\end{equation}
where $c_n>0$ depends only on $n$.

\end{lemma}

\begin{proof}

\medskip
\noindent\textit{Step 1.
Reduction to a Fourier average.}

For fixed $a,\tau$, write $K$ for the kernel of $\mathcal{Q}(a,\tau)$.
The Hilbert--Schmidt norm formula gives

\[
\|\mathcal{Q}(a,\tau)\|_{\mathfrak S_2}^2
=\int_{\mathbb{S}(a^\perp)}\int_{\mathbb{S}(a^\perp)}
|K(b,c)|^2\,\mathrm{d}\sigma_a(c)\,\mathrm{d}\sigma_a(b).
\]

Choose $a$ with distribution $\mu$ from Definition~\ref{def:direction-average}.
Given $a$, choose $b,c$ independently with distribution $\sigma_a$ from Definition~\ref{def:incident-response}.
Write $\mathbb E_{a,b,c}$ for this joint average.
The kernel formula \eqref{eq:Q-kernel} gives

\[
\mathbb E_a\|\mathcal{Q}(a,\tau)\|_{\mathfrak S_2}^2
=(2\pi)^n\mathbb E_{a,b,c}|\widehat{q_1-q_2}(\tau(b-c))|^2.
\]

\medskip
\noindent\textit{Step 2.
Density of the Fourier frequency.}
To compute the density of $\tau(b-c)$ in the average from Step~1, we first find the distribution of $|b-c|$ and then use rotational invariance.

For fixed $a,b$, parametrize $c\in\mathbb{S}(a^\perp)$ by $c=(b\cdot c)b+\sqrt{1-(b\cdot c)^2}\,\omega$, where $\omega\in\mathbb{S}(a^\perp\cap b^\perp)$.
The surface-element calculation used in Step~2 of Lemma~\ref{lem:source-average} gives $(1-(b\cdot c)^2)^{(n-4)/2}\,\mathrm{d}(b\cdot c)\,\mathrm{d}S(\omega)$ in these coordinates.
Integrating in $\omega$ and normalizing gives the density of $b\cdot c$
\begin{equation}\label{eq:coordinate-density}
\kappa_n(1-(b\cdot c)^2)^{(n-4)/2},
\qquad -1<b\cdot c<1,
\end{equation}
with respect to $\mathrm{d}(b\cdot c)$, where $\kappa_n=|\mathbb{S}^{n-3}|/|\mathbb{S}^{n-2}|=\Gamma((n-1)/2)/(\sqrt\pi\,\Gamma((n-2)/2))$.
To obtain the distance density from \eqref{eq:coordinate-density}, use $b\cdot c=1-|b-c|^2/2$ and $|\mathrm{d}(b\cdot c)|=|b-c|\,\mathrm{d}(|b-c|)$.
The density of $|b-c|$ is
\[
\kappa_n|b-c|^{n-3}(1-|b-c|^2/4)^{(n-4)/2},
\qquad 0<|b-c|<2,
\]
with respect to $\mathrm{d}(|b-c|)$.
This density is the same for every fixed $a,b$.

The joint distribution of $(a,b,c)$ is invariant under simultaneous rotations.
Thus, conditional on $|b-c|$, the direction of $b-c$ is uniform on $\mathbb{S}^{n-1}$.
The spherical shell at distance $|b-c|$ from the origin has volume element $|\mathbb{S}^{n-1}|\,|b-c|^{n-1}\,\mathrm{d}(|b-c|)$.
Dividing the distance density by $|\mathbb{S}^{n-1}|\,|b-c|^{n-1}$ gives the density with respect to $n$-dimensional volume, with the factor $|b-c|^{-2}$.
There is no mass at the origin since $b=c$ has probability zero.
Scaling $b-c$ by $\tau$ contributes a Jacobian factor $\tau^{-n}$, while the inverse-square factor contributes $\tau^2$.
Thus $\tau(b-c)$ has density
\begin{equation}\label{eq:difference-density}
\frac{\kappa_n}{|\mathbb{S}^{n-1}|}\tau^{-(n-2)}|\xi|^{-2}
\left(1-\frac{|\xi|^2}{4\tau^2}\right)^{(n-4)/2}
\mathbf{1}_{\{0<|\xi|<2\tau\}}
\end{equation}
with respect to Lebesgue measure $\mathrm{d}\xi$ on $\mathbb{R}^n$.

\medskip
\noindent\textit{Step 3.
Comparison with the $\dot H^{-1}$ norm.}
The factor $|\xi|^{-2}$ in this density is the weight defining the $\dot H^{-1}$ norm.

Integrating $|\widehat{q_1-q_2}(\xi)|^2$ against \eqref{eq:difference-density} in the average from Step~1 gives
\begin{equation}\label{eq:response-average-identity}
\mathbb E_a\|\mathcal{Q}(a,\tau)\|_{\mathfrak S_2}^2
=c_n\tau^{-(n-2)}
\int_{|\xi|<2\tau}\frac{|\widehat{q_1-q_2}(\xi)|^2}{|\xi|^2}
\left(1-\frac{|\xi|^2}{4\tau^2}\right)^{(n-4)/2}\,\mathrm{d}\xi,
\end{equation}

where
\[
c_n=\frac{(2\pi)^n\kappa_n}{|\mathbb{S}^{n-1}|}.
\]
For $n\geq4$, we use $(1-|\xi|^2/(4\tau^2))^{(n-4)/2}\leq1$ on $|\xi|<2\tau$ and extend the integral in \eqref{eq:response-average-identity} to $\mathbb{R}^n$ to obtain \eqref{eq:response-average}.

The regularity and support assumptions give $q_1-q_2\in H^{-1}_E$, so the homogeneous norm is finite by Lemma~\ref{lem:gamma-decay}.

\end{proof}

The Hilbert--Schmidt decay in Lemma~\ref{lem:response-average} holds without assuming equality of the boundary maps and therefore does not by itself imply $q_1=q_2$.

\begin{remark}\label{rem:dimension-estimates}
In Lemmas~\ref{lem:source-average} and~\ref{lem:response-average}, the condition $n\geq4$ bounds the weights in \eqref{eq:source-denominator-average} and \eqref{eq:response-average-identity} by one.
For $n=3$, both weights have exponent $-1/2$ and become unbounded at the boundary, so they cannot be discarded in these upper bounds.
\end{remark}

\subsection{From finite codimension to trace norm decay}\label{sec:trace}

\begin{proof}[Proof of Proposition~\ref{prop:trace-decay}]
Let $n\geq4$ and let $h_1,h_2$ satisfy \eqref{eq:compact-model}.
Let $q_1,q_2$ be their associated potentials from Definition~\ref{def:schrodinger-potentials}, and let $\mathcal{Q}(a,\tau)$ be the free response of $q_1-q_2$ defined by \eqref{eq:Q-form}.

\medskip
\noindent\textit{Step 1.
From the boundary data to a trace norm bound.}
For $\tau\geq\max\{\tau_0(q_1),\tau_0(q_2)\}$ and $a\in\mathbb{S}^{n-1}$, let $\mathcal{P}$ project orthogonally onto the common compatible weight space $\mathcal{G}$ of Definition~\ref{def:compatible-weights}.
By the equality of the Dirichlet-to-Neumann maps in \eqref{eq:compact-model}, Lemma~\ref{lem:nonlinear} gives the compression bound \eqref{eq:compatible-trace-bound}.

To control the full response through the decomposition \eqref{eq:response-decomposition}, we must also estimate the remainder.
Its rank is controlled by the same codimension, since
\[
\mathcal{Q}-\mathcal{P}\mathcal{Q}\mathcal{P}
=(I-\mathcal{P})\mathcal{Q}+\mathcal{P}\mathcal{Q}(I-\mathcal{P}).
\]
Each summand has rank at most $\operatorname{rank}(I-\mathcal{P})$, so the remainder has rank at most $2\operatorname{rank}(I-\mathcal{P})$.

To bound its Hilbert--Schmidt norm, note that the two summands have orthogonal ranges in $\ker\mathcal{P}$ and $\operatorname{ran}\mathcal{P}$.
Their squared norms therefore add, and right multiplication by the orthogonal projection $I-\mathcal{P}$ does not increase this norm:
\[
\begin{aligned}
\|\mathcal{Q}-\mathcal{P}\mathcal{Q}\mathcal{P}\|_{\mathfrak S_2}^2
&=\|(I-\mathcal{P})\mathcal{Q}\|_{\mathfrak S_2}^2
+\|\mathcal{P}\mathcal{Q}(I-\mathcal{P})\|_{\mathfrak S_2}^2\\
&\leq\|(I-\mathcal{P})\mathcal{Q}\|_{\mathfrak S_2}^2
+\|\mathcal{P}\mathcal{Q}\|_{\mathfrak S_2}^2
=\|\mathcal{Q}\|_{\mathfrak S_2}^2.
\end{aligned}
\]
By the norm formulas in Definition~\ref{def:trace-ideals} (see also \cite[Chapter~2]{simon2005trace}), Cauchy--Schwarz for the nonzero singular values converts these rank and Hilbert--Schmidt bounds into
\begin{equation}\label{eq:remainder-trace-bound}
\|\mathcal{Q}-\mathcal{P}\mathcal{Q}\mathcal{P}\|_{\mathfrak S_1}
\leq\sqrt{2\operatorname{rank}(I-\mathcal{P})}\,\|\mathcal{Q}\|_{\mathfrak S_2}.
\end{equation}

Recall from Proposition~\ref{prop:compatible-solutions} that $\Gamma_{12}(\tau)=\Gamma_{q_1}(\tau)+\Gamma_{q_2}(\tau)$.
Combining the decomposition \eqref{eq:response-decomposition}, the trace norm bounds \eqref{eq:compatible-trace-bound} and \eqref{eq:remainder-trace-bound}, and the uniform codimension bound \eqref{eq:codimension} gives
\[
\begin{aligned}
\|\mathcal{Q}\|_{\mathfrak S_1}
&\leq\|\mathcal{P}\mathcal{Q}\mathcal{P}\|_{\mathfrak S_1}
+\|\mathcal{Q}-\mathcal{P}\mathcal{Q}\mathcal{P}\|_{\mathfrak S_1}\\
&\leq4\sum_{j=1}^2\|\mathcal{C}_j\|_{\mathfrak S_2}^2
+\sqrt{2\operatorname{rank}(I-\mathcal{P})}\,\|\mathcal{Q}\|_{\mathfrak S_2}\\
&\leq4\sum_{j=1}^2\|\mathcal{C}_j\|_{\mathfrak S_2}^2
+C\tau^{(n-2)/2}\sqrt{\Gamma_{12}(\tau)}\,\|\mathcal{Q}\|_{\mathfrak S_2}.
\end{aligned}
\]
Here $C$ is independent of $a$ and $\tau$.

\medskip
\noindent\textit{Step 2.
Spherical averaging and decay.}
To apply the spherical estimates, note that \eqref{eq:compact-model} gives $\nabla h_j\in L^n_E$ and $|\nabla h_j|^2\in L^{n/2}_E$ for the potentials $q_j=\Delta h_j+|\nabla h_j|^2$.
Since $n\geq4$ and $E$ is compact, both coefficients also belong to $L^2$, so $q_j\in H^{-1}_E$.

For fixed $\tau$, the norm functions in the following average are measurable.
Locally, orthogonal rotations identify $\mathbb S(a^\perp)$ with a fixed sphere; by \eqref{eq:Q-kernel}, the resulting kernels vary continuously in $L^2$, so $\mathcal Q(a,\tau)$ varies continuously in $\mathfrak S_2$.
Each singular value is continuous in the operator norm, and the trace norm is the supremum of its finite singular-value sums (Definition~\ref{def:trace-ideals}; see also \cite[Chapter~2]{simon2005trace}).
Thus $a\mapsto\|\mathcal Q(a,\tau)\|_{\mathfrak S_1}$ is measurable; the source norms are measurable by \eqref{eq:source-integral} and Tonelli's theorem.
Averaging the bound at fixed $\tau$ with respect to the probability measure $\mu$ and applying Cauchy--Schwarz to the free response term gives the first inequality below.
For the second inequality, apply Lemma~\ref{lem:source-average} with $q=q_j$ and sum the upper bounds in \eqref{eq:source-average} to bound the source term by $C\Gamma_{12}(\tau)$.

Taking the square root of \eqref{eq:response-average} in Lemma~\ref{lem:response-average} bounds the free response root mean square by $\sqrt{c_n}\tau^{-(n-2)/2}\|q_1-q_2\|_{\dot H^{-1}}$.

\begin{equation}\label{eq:averaged-trace}
\begin{split}
\mathbb E_a\|\mathcal{Q}(a,\tau)\|_{\mathfrak S_1}
&\leq 4\sum_{j=1}^2\mathbb E_a\|\mathcal{C}_j\|_{\mathfrak S_2}^2
+C\tau^{(n-2)/2}\sqrt{\Gamma_{12}(\tau)}
\big(\mathbb E_a\|\mathcal{Q}(a,\tau)\|_{\mathfrak S_2}^2\big)^{1/2}\\
&\leq C\Gamma_{12}(\tau)
+C\sqrt{\Gamma_{12}(\tau)}\,
\tau^{(n-2)/2}\tau^{-(n-2)/2}\|q_1-q_2\|_{\dot H^{-1}}\\
&=C\Gamma_{12}(\tau)
+C\sqrt{\Gamma_{12}(\tau)}\,\|q_1-q_2\|_{\dot H^{-1}}.
\end{split}
\end{equation}
Lemma~\ref{lem:gamma-decay} gives $\|q_1-q_2\|_{\dot H^{-1}}<\infty$ and $\Gamma_{12}(\tau)\to0$, which proves \eqref{eq:Ea}.
\end{proof}

\section{Fourier recovery and uniqueness}\label{sec:uniqueness}
We first show that the averaged trace norm decay in Proposition~\ref{prop:trace-decay} forces the potential difference $q_1-q_2$ to vanish.
A weighted energy identity then gives $\gamma_1=\gamma_2$, completing the proof of Theorem~\ref{thm:main}.

\subsection{Recovery of every Fourier frequency}\label{sec:frequencies}
Section~\ref{sec:trace-method} establishes the averaged trace norm decay of $\mathcal{Q}$ stated in Proposition~\ref{prop:trace-decay}, under equality of the Dirichlet-to-Neumann maps.
It remains to deduce equality of the potentials from this decay.
The following proposition supplies this implication by proving that $\widehat{q_1-q_2}$ vanishes at every frequency.

\begin{proposition}\label{prop:frequency}
Let $n\geq3$, and let $h_1,h_2$ satisfy the regularity and support conditions in \eqref{eq:compact-model}, with associated potentials $q_1,q_2$ from Definition~\ref{def:schrodinger-potentials}.
Let $\mathcal{Q}(a,\tau)$ be the free response operator associated with $q_1-q_2$ as in Definition~\ref{def:incident-response}.
With $\mathbb E_a$ denoting the normalized spherical average from Definition~\ref{def:direction-average}, suppose that
\[
\mathbb E_a\|\mathcal{Q}(a,\tau)\|_{\mathfrak S_1}\longrightarrow0\qquad\text{as }\tau\longrightarrow\infty.
\]
Then $q_1=q_2$ as distributions.
\end{proposition}

\begin{proof}
The compact support of $q_1-q_2$ makes $\widehat{q_1-q_2}$ smooth.

\medskip
\noindent\textit{Step 1.
Decay for almost every direction.}
We first choose a sequence along which the trace norm tends to zero for almost every fixed direction.
The hypothesis allows us to choose $\tau_j\to\infty$ with $\mathbb E_a\|\mathcal{Q}(a,\tau_j)\|_{\mathfrak S_1}\leq2^{-j}$.
The summands below are nonnegative and measurable in $a$ (Subsection~\ref{sec:trace}).
Thus Tonelli's theorem~\cite[Theorem~2.37(a)]{folland1999real} allows us to exchange summation and integration with respect to the measure $\mu$ from Definition~\ref{def:direction-average}:
\[
\int_{\mathbb{S}^{n-1}}\sum_{j=1}^{\infty}\|\mathcal{Q}(a,\tau_j)\|_{\mathfrak S_1}\,\mathrm{d}\mu(a)
=\sum_{j=1}^{\infty}\mathbb E_a\|\mathcal{Q}(a,\tau_j)\|_{\mathfrak S_1}\leq1.
\]
Since this integral is finite, the series converges for $\mu$-almost every $a$.
Let $A$ be the set of these directions.
For each $a\in A$, the terms of the convergent series tend to zero:
\[
\|\mathcal{Q}(a,\tau_j)\|_{\mathfrak S_1}\longrightarrow0
\qquad\text{as }j\longrightarrow\infty.
\]
The sequence $(\tau_j)$ and the full-measure set $A$ will remain fixed throughout the proof.

\medskip
\noindent\textit{Step 2.
Recovery on a fixed hyperplane.}
Fix $a\in A$ and a nonzero Fourier frequency $\xi\in a^\perp$.
We aim to prove $\widehat{q_1-q_2}(\xi)=0$.
The kernel in \eqref{eq:Q-kernel} samples the frequency $\tau(b-c)$, so its diagonal $b=c$ gives only zero frequency.
We therefore rotate one direction to reach the chosen $\xi$.
Since $n\geq3$, choose a unit vector $b_0$ orthogonal to both $a$ and $\xi$, and define
\[
Mx=\xi(b_0\cdot x)-b_0(\xi\cdot x),\qquad x\in\mathbb{R}^n.
\]
The real matrix $M$ is skew-symmetric: $M^\top=-M$, where $\top$ denotes transpose.
Together with $Ma=0$, this makes $\exp(M/\tau)$ a rotation preserving $\mathbb{S}(a^\perp)$.
The identity $Mb_0=\xi$ gives $\tau(\exp(M/\tau)b_0-b_0)\to\xi$.
To obtain a nonnegative integral in the limit, define a test operator on $L^2(\mathbb{S}(a^\perp),\mathrm{d}\sigma_a)$ by
\begin{equation}\label{eq:trace-test}
(\mathcal{Z}_{\tau,a}g)(b)
=\overline{\widehat{q_1-q_2}(Mb)}\,g\bigl(\exp(M/\tau)b\bigr),\qquad b\in\mathbb{S}(a^\perp).
\end{equation}
The Fourier factor is bounded on the sphere and independent of $\tau$.
Rotation preserves normalized surface measure, so $\mathcal{Z}_{\tau,a}$ is uniformly bounded in $\tau$.

Since $\mathcal{Q}(a,\tau_j)$ is trace class and $\mathcal{Z}_{\tau_j,a}$ is bounded, their product is trace class.
The trace norm inequality~\cite[Chapters~2--3]{simon2005trace} and Step~1 give
\begin{equation}\label{eq:trace-duality}
\begin{aligned}
\bigl|\operatorname{Tr}(\mathcal{Z}_{\tau_j,a}\mathcal{Q}(a,\tau_j))\bigr|
&\leq\|\mathcal{Z}_{\tau_j,a}\|\,\|\mathcal{Q}(a,\tau_j)\|_{\mathfrak S_1}\\
&\leq\left(\sup_{b\in\mathbb{S}(a^\perp)}\bigl|\widehat{q_1-q_2}(Mb)\bigr|\right)
\|\mathcal{Q}(a,\tau_j)\|_{\mathfrak S_1}\longrightarrow0.
\end{aligned}
\end{equation}

To compute the same trace, we use the original kernel at $(\exp(M/\tau)b,b)$, together with the Fourier factor in \eqref{eq:trace-test}.
The smooth-kernel trace formula~\cite[Corollary~4.4]{delgado2014schatten} gives
\begin{equation}\label{eq:rotated-trace}
\begin{aligned}
\operatorname{Tr}(\mathcal{Z}_{\tau,a}\mathcal{Q}(a,\tau))
&=(2\pi)^{n/2}\int_{\mathbb{S}(a^\perp)}
\overline{\widehat{q_1-q_2}(Mb)}\\
&\qquad\times\widehat{q_1-q_2}\bigl(\tau(\exp(M/\tau)-I)b\bigr)\,\mathrm{d}\sigma_a(b).
\end{aligned}
\end{equation}
On the sphere, $\tau(\exp(M/\tau)-I)b\to Mb$ uniformly.
Since $\widehat{q_1-q_2}$ is smooth, the integrand in \eqref{eq:rotated-trace} converges uniformly to $|\widehat{q_1-q_2}(Mb)|^2$.
Taking $\tau=\tau_j\to\infty$ and using \eqref{eq:trace-duality} gives
\begin{equation}\label{eq:rotation-limit}
0=\lim_{j\to\infty}\operatorname{Tr}(\mathcal{Z}_{\tau_j,a}\mathcal{Q}(a,\tau_j))
=(2\pi)^{n/2}\int_{\mathbb{S}(a^\perp)}
\bigl|\widehat{q_1-q_2}(Mb)\bigr|^2\,\mathrm{d}\sigma_a(b).
\end{equation}
The continuous nonnegative integrand must vanish everywhere on the sphere.
At $b_0$, where $Mb_0=\xi$, this yields $\widehat{q_1-q_2}(\xi)=0$.
As $\xi$ was arbitrary and the Fourier transform is continuous at zero, it vanishes on $a^\perp$ for every $a\in A$.

\medskip
\noindent\textit{Step 3.
Recovery of all frequencies.}
Since $A$ has full $\mu$-measure and every nonempty open subset of $\mathbb{S}^{n-1}$ has positive $\mu$-measure, $A$ is dense in the sphere.
Fix $\xi\in\mathbb{R}^n$ and choose a unit vector $a_0$ orthogonal to $\xi$.
By density, there is a sequence $a_m\in A$ with $a_m\to a_0$.
The orthogonal projections of $\xi$ onto $a_m^\perp$ satisfy
\[
\xi-(\xi\cdot a_m)a_m\in a_m^\perp,
\qquad
\xi-(\xi\cdot a_m)a_m\longrightarrow \xi.
\]
By Step~2, $\widehat{q_1-q_2}$ vanishes at each of these projected frequencies.
Continuity then gives $\widehat{q_1-q_2}(\xi)=0$.
Since $\xi$ was arbitrary, injectivity of the Fourier transform yields $q_1=q_2$.
\end{proof}

\begin{remark}
In Step~2 of Proposition~\ref{prop:frequency}, the condition $n\geq3$ ensures that $a^\perp\cap\xi^\perp$ contains a unit vector $b_0$ for each nonzero $\xi\in a^\perp$.
The vectors $b_0$ and $\xi$ span a plane inside $a^\perp$, allowing the small rotations used to recover $\xi$.
For $n=2$, $\mathbb{S}(a^\perp)$ consists of only two points, so this construction is unavailable.
The condition $n\geq4$ is used in the low-frequency row estimate of Proposition~\ref{prop:row} and in the averaged estimates of Section~\ref{sec:trace-method}; see Appendix~\ref{app:auxiliary} and Remark~\ref{rem:dimension-estimates}.
\end{remark}

\subsection{Proof of the main theorem}\label{sec:recovery}
\begin{proof}[Proof of Theorem~\ref{thm:main}]
Let $n\geq4$ and let $\Omega\subset\mathbb{R}^n$ be a bounded Lipschitz domain.
Suppose that $\gamma_1,\gamma_2\in W^{1,n}(\Omega;\mathbb{R})$ satisfy the common ellipticity bounds \eqref{eq:ellipticity} and that $\Lambda_{\gamma_1,\Omega}=\Lambda_{\gamma_2,\Omega}$.
Lemma~\ref{lem:gluing} gives a ball $B$ with $\Omega\Subset B$ and uniformly elliptic extensions $\widetilde\gamma_j$ with $\widetilde\gamma_j-1\in W^{1,n}(\mathbb{R}^n)$.
The extensions agree outside $\Omega$ and equal one outside a compact subset of $B$.
The equality of the boundary maps passes to the larger ball, so
\[
\Lambda_{\widetilde\gamma_1,B}=\Lambda_{\widetilde\gamma_2,B}.
\]
Write $\gamma_j$ for the extensions and set $h_j=\tfrac12\log\gamma_j$.
The logarithms are bounded, belong to $W^{1,n}(\mathbb{R}^n)$, and are supported in a common compact set $E\Subset B$, so they satisfy \eqref{eq:compact-model}.

For the associated potentials $q_1,q_2$ of Definition~\ref{def:schrodinger-potentials}, Alessandrini's identity~\eqref{eq:alessandrini} applies and gives
\[
\langle q_1-q_2,w_1w_2\rangle=0
\]
for every pair of weak solutions $w_j\in H^1(B)$ of $(-\Delta+q_j)w_j=0$ in $B$.

The logarithms have $W^{1,n}$ regularity, so the associated potentials have the form $q_j=\operatorname{div}(\nabla h_j)+|\nabla h_j|^2$, with $\nabla h_j\in L^n_E$ and $|\nabla h_j|^2\in L^{n/2}_E$.
In dimension $n\geq4$, these are the hypotheses for the partial inverse construction in Proposition~\ref{prop:partial} and the averaged trace norm estimate in Proposition~\ref{prop:trace-decay}.
Subsection~\ref{sec:operators} writes the equations for the CGO corrections on $L^2$.
For sufficiently large $\tau$, the partial inverses in Proposition~\ref{prop:partial} give corrections with residuals in finite-dimensional subspaces.
The compatibility conditions~\eqref{eq:compatible} eliminate these residuals, giving the exact solutions in Proposition~\ref{prop:compatible-solutions}.
Let $\mathcal P$ be the orthogonal projection onto the common compatible subspace.
Applying Alessandrini's identity to these solutions and using Lemma~\ref{lem:nonlinear} gives the compressed response identity~\eqref{eq:compressed-response}.
The uniform bounds for the partial inverses then give the trace norm estimate~\eqref{eq:compatible-trace-bound},
\[
\|\mathcal P\mathcal Q\mathcal P\|_{\mathfrak S_1}
\leq 4\sum_{j=1}^2\|\mathcal C_j\|_{\mathfrak S_2}^2,
\]
where $\mathcal C_j$ are the source operators in \eqref{eq:source}.

By \eqref{eq:source-average} and the decay $\Gamma_{q_j}(\tau)\to0$ from Lemma~\ref{lem:gamma-decay}, the right-hand side tends to zero after averaging over $a$.
Subsection~\ref{sec:trace} controls the contribution of $\mathcal Q-\mathcal P\mathcal Q\mathcal P$ by combining the codimension bound~\eqref{eq:codimension} with the averaged Hilbert--Schmidt estimate~\eqref{eq:response-average}.
Together these estimates give the averaged trace norm decay~\eqref{eq:Ea},
\[
\mathbb E_a\|\mathcal{Q}(a,\tau)\|_{\mathfrak S_1}\longrightarrow0
\qquad\text{as }\tau\longrightarrow\infty.
\]
By Proposition~\ref{prop:frequency}, this decay forces the Fourier transform of $q_1-q_2$ to vanish at every frequency, so $q_1=q_2$ as distributions.

Subtracting the formulas for $q_1$ and $q_2$ gives the equation for the logarithmic difference,
\begin{equation}\label{eq:difference-potential}
0=q_1-q_2=\Delta(h_1-h_2)+\nabla(h_1+h_2)\cdot\nabla(h_1-h_2).
\end{equation}
The drift term belongs to $L^{n/2}(B)\subset L^2(B)$ by H\"older's inequality and $n\geq4$.
Since $h_1-h_2\in H^1(B)$, both terms in \eqref{eq:difference-potential} act continuously on $H^1_0(B)$, so the equation holds against all such tests.

Set $\gamma_0=\exp(h_1+h_2)=\sqrt{\gamma_1\gamma_2}$.
The Sobolev chain and product rules give $\gamma_0(h_1-h_2)\in W^{1,n}(B)$ and $\nabla\gamma_0=\gamma_0\nabla(h_1+h_2)$.
Both $h_1-h_2$ and $\gamma_0(h_1-h_2)$ have compact support in $B$ and hence belong to $H^1_0(B)$.
Testing \eqref{eq:difference-potential} with $\gamma_0(h_1-h_2)$, the term involving $\nabla\gamma_0$ cancels the drift term, giving
\begin{equation}\label{eq:final-energy}
\begin{aligned}
0&=\int_B\nabla(h_1-h_2)\cdot\nabla\bigl(\gamma_0(h_1-h_2)\bigr)\,\mathrm{d}x\\
&\quad-\int_B\gamma_0(h_1-h_2)\,\nabla(h_1+h_2)\cdot\nabla(h_1-h_2)\,\mathrm{d}x\\
&=\int_B\gamma_0|\nabla(h_1-h_2)|^2\,\mathrm{d}x.
\end{aligned}
\end{equation}
The positive lower bound on $\gamma_0$ gives $\nabla(h_1-h_2)=0$ almost everywhere in $B$.
Poincar\'e's inequality on $H^1_0(B)$ then gives $h_1=h_2$ almost everywhere in $B$.
Consequently, $\gamma_1=\exp(2h_1)=\exp(2h_2)=\gamma_2$ almost everywhere on the original domain $\Omega$.
\end{proof}

\appendix
\section{Proof of the low-frequency row estimate}\label{app:auxiliary}
We begin with the local geometry of the symbol $d_a$.

Write $d_{a,1}$ for the symbol in \eqref{eq:symbol-intro} at $\tau=1$, so that $d_a(\tau\eta)=\tau^2d_{a,1}(\eta)$.
At unit scale put $\Sigma_a=\{\xi:a\cdot\xi=0,\ |\xi|=1\}$ and let $N_s(\Sigma_a)$ be its open Euclidean $s$-neighborhood.
Write $\xi=xa+\rho\theta$, with $\theta\in\mathbb{S}(a^\perp)$.
Then $\delta=\operatorname{dist}(\xi,\Sigma_a)=(x^2+(\rho-1)^2)^{1/2}$ and
\[
d_{a,1}(\xi)=2(\rho-1)+(\rho-1)^2+x^2-2\mathrm{i}x.
\]
For $\delta<s_0<1$, the triangle inequality gives
\begin{equation}\label{eq:normal-symbol}
(2-s_0)\delta\leq|d_{a,1}(\xi)|\leq(2+s_0)\delta.
\end{equation}
The volume element is $\rho^{n-2}\,\mathrm{d}x\,\mathrm{d}\rho\,\mathrm{d}S(\theta)$.
The two normal coordinates give $|N_s(\Sigma_a)|\leq C_ns^2$ for $s<1/8$, and $|d_{a,1}|^{-1}$ is locally integrable.

We prove Proposition~\ref{prop:row} through a geometric bound for translated characteristic tubes and the resulting convolution estimate.

\begin{lemma}\label{lem:tubes}
For $n\geq4$, $0<s\leq t\leq s_0<1/8$, and $v\neq0$,
\begin{equation}\label{eq:tubes}
|N_s(\Sigma_a)\cap(v+N_t(\Sigma_a))|
\leq C_n s^2\min\{1,t/|v|\}.
\end{equation}
The bound is uniform in $a$.
\end{lemma}

\begin{proof}
Write $v=\lambda a+v_\perp$ with $v_\perp\perp a$, and set $r=|v|$.
If the tubes meet, the triangle inequality gives $r\leq2+s+t$.
When $|\lambda|\geq r/2$, their normal coordinates require $|\lambda|\leq s+t\leq2t$, hence $r\leq4t$.
The single-tube volume bound then proves \eqref{eq:tubes}.

Otherwise $|v_\perp|\geq\sqrt3r/2$.
For $\xi=xa+\rho\theta$ in the intersection, the radial tube conditions imply
\[
\big||\xi|^2-1\big|\leq Cs,\qquad
\big||\xi-v|^2-1\big|\leq Ct.
\]
Subtracting and using $s\leq t$ gives $|r^2-2\xi\cdot v|\leq Ct$.
Since $|\xi-\theta|\leq Cs$ and $r$ is bounded,
\[
|2\xi\cdot v-2\theta\cdot v_\perp|\leq Csr\leq Ct.
\]
Thus $|2\theta\cdot v_\perp-r^2|\leq Ct$, which confines $\theta\cdot v_\perp/|v_\perp|$ to an interval of length at most $Ct/r$.
By rotational invariance, \eqref{eq:coordinate-density} applies to $\theta\cdot v_\perp/|v_\perp|$; for $n\geq4$, this density is bounded up to the endpoints.
The permissible directions therefore have area at most $C\min\{1,t/r\}$.
Integrating the two normal coordinates supplies $Cs^2$ and proves \eqref{eq:tubes}, including tangential intersections.
\end{proof}

\begin{lemma}\label{lem:convolution}
Fix $R>2$ and $n\geq4$.
For every $a\in\mathbb{S}^{n-1}$, $\tau\geq2$, and $v\neq0$,
\begin{equation}\label{eq:convolution}
I_{a,\tau}(v):=
\int_{|\xi|\leq R\tau}
\frac{\mathrm{d}\xi}{|d_a(\xi)|\,|d_a(\xi-v)|}
\leq\frac{C_{n,R}\tau^{n-2}}{|v|^2+\tau|v|}.
\end{equation}
\end{lemma}

\begin{proof}
At unit scale, \eqref{eq:normal-symbol} bounds the inverse symbol on bounded sets by a bounded function plus
\[
H_a(\xi)=\delta(\xi)^{-1}\mathbf{1}_{\{\delta(\xi)<s_0\}},
\qquad \delta(\xi)=\operatorname{dist}(\xi,\Sigma_a).
\]
The tube volume bound gives $\|H_a\|_1\leq Cs_0$, and
\[
H_a=s_0^{-1}\mathbf{1}_{N_{s_0}(\Sigma_a)}
+\int_0^{s_0}\mathbf{1}_{N_s(\Sigma_a)}\frac{\mathrm{d}s}{s^2}.
\]
Terms involving at most one singular factor, and terms with a fixed tube width, are uniformly bounded.
For the two variable widths, apply Tonelli's theorem and Lemma~\ref{lem:tubes} on the regions $s\leq t$ and $t\leq s$.
Integrating first in the smaller width gives
\begin{equation}\label{eq:log-convolution}
\int H_a(\xi)H_a(\xi-v)\,\mathrm{d}\xi
\leq C+C\int_0^{s_0}\frac{\min\{1,t/|v|\}}{t}\,\mathrm{d}t
\leq C\left(1+\frac{s_0}{|v|}\right).
\end{equation}

Scaling gives $I_{a,\tau}(v)=\tau^{n-4}I_{a,1}(v/\tau)$.
For $0<|v|\leq2(R+1)\tau$, \eqref{eq:log-convolution} implies
\[
I_{a,\tau}(v)
\leq C_{n,R}\tau^{n-4}\left(1+\frac{\tau}{|v|}\right)
\leq\frac{C_{n,R}\tau^{n-2}}{|v|^2+\tau|v|}.
\]
For $|v|>2(R+1)\tau$, one has $|d_a(\xi-v)|\geq C_R^{-1}|v|^2$ on $|\xi|\leq R\tau$, while
\[
\int_{|\xi|\leq R\tau}|d_a(\xi)|^{-1}\,\mathrm{d}\xi
=\tau^{n-2}\int_{|\eta|\leq R}|d_{a,1}(\eta)|^{-1}\,\mathrm{d}\eta
\leq C_{n,R}\tau^{n-2}.
\]
This gives \eqref{eq:convolution} in the remaining region.
The convolution may be infinite at $v=0$; this single point does not affect the subsequent frequency integral.
\end{proof}

\begin{proof}[Proof of Proposition~\ref{prop:row}]
Let $z_1,z_2$ have Fourier transforms that are smooth, compactly supported, and supported away from the zero set of $d_a$.
Then $\mathcal{A}z_1$ and $\mathcal{A}z_2$ are Schwartz functions, and
\[
\widehat{q\mathcal{A}z_1}(\xi)=(2\pi)^{-n/2}\int
\widehat q(\xi-\eta)|d_a(\eta)|^{-1/2}\widehat{z_1}(\eta)\,\mathrm{d}\eta.
\]
It follows in pairing with $z_2$ that the frequency kernel of $\mathcal{L}_\tau \mathcal{K}_q$ is
\[
K_{\mathrm{row}}(\xi,\eta)=(2\pi)^{-n/2}
\frac{\mathbf{1}_{\{|\xi|\leq R\tau\}}\widehat q(\xi-\eta)}
{|d_a(\xi)|^{1/2}|d_a(\eta)|^{1/2}}.
\]
Tonelli and $v=\xi-\eta$ give
\[
\iint|K_{\mathrm{row}}(\xi,\eta)|^2\,\mathrm{d}\eta\,\mathrm{d}\xi
=(2\pi)^{-n}\int|\widehat q(v)|^2I_{a,\tau}(v)\,\mathrm{d}v
\leq C\tau^{n-2}\Gamma_q(\tau).
\]
The $L^2$ kernel criterion identifies its squared Hilbert--Schmidt norm with this integral~\cite[Theorem~2.11]{simon2005trace}.
The kernel operator agrees with $\mathcal{L}_\tau \mathcal{K}_q$ on the dense Fourier test class, hence on all of $L^2$ by boundedness.
\end{proof}

\section*{Acknowledgments}
L.~Chen was supported in part by the National Natural Science Foundation of China (Grant No.~12622108), Beijing Institute of Technology (Grant No.~2022CX01002), and the Hebei Natural Science Foundation (Grant No.~A2025105003).
Y.~Jiang was supported in part by the National Key R\&D Program of China (Grant No.~2024YFA1012302) and the Hong Kong Research Grants Council (Project No.~JRFS2627-1S06).
H.~Liu was supported by the General Research Fund of the Hong Kong Research Grants Council (Project Nos.~11311122, 11304224, and 11303125).

\section*{Generative AI disclosure}
The authors acknowledge the use of ChatGPT (OpenAI) for exploratory discussions and suggestions on the formulation and organization of parts of the argument in Sect.~3 that use standard tools from operator theory.
OpenAI Codex assisted with drafting and revising the English prose, organizing the manuscript, and finding and checking references.
The authors independently checked and finalized all mathematical statements, proofs, and references and take full responsibility for the content of the paper.

\bibliographystyle{plain}
\bibliography{w1n_ref}
\end{document}